\documentclass[11pt]{amsart}

\usepackage{amsmath}
\usepackage{amssymb}
\newtheorem{theorem}{Theorem}[section]
\newtheorem{claim}[theorem]{Claim}

\newtheorem{lemma}[theorem]{Lemma}

\newtheorem{proposition}[theorem]{Proposition}
\newtheorem{corollary}[theorem]{Corollary}

\theoremstyle{definition}
\newtheorem{definition}[theorem]{Definition}

\newtheorem{question}[theorem]{Question}

\theoremstyle{remark}
\newtheorem{remark}[theorem]{Remark}

\newtheorem{conjecture}[theorem]{Conjecture}

\newcount\skewfactor
\def\mathunderaccent#1#2 {\let\theaccent#1\skewfactor#2
\mathpalette\putaccentunder}
\def\putaccentunder#1#2{\oalign{$#1#2$\crcr\hidewidth
\vbox to.2ex{\hbox{$#1\skew\skewfactor\theaccent{}$}\vss}\hidewidth}}
\def\name{\mathunderaccent\tilde-3 }

\def\smallbox#1{\leavevmode\thinspace\hbox{\vrule\vtop{\vbox
   {\hrule\kern1pt\hbox{\vphantom{\tt/}\thinspace{\tt#1}\thinspace}}
   \kern1pt\hrule}\vrule}\thinspace}
\theoremstyle{definition}

\newtheorem{fact}[theorem]{Fact}

\theoremstyle{plain}

\newcommand{\ZFC}{\mathrm{ZFC}}
\newcommand{\ZF}{\mathrm{ZF}}
\newcommand{\AD}{\mathrm{AD}}

\newcommand{\fr}{{}^\frown}

\newcommand{\force}{\Vdash}

\newcommand{\la}{\langle}
\newcommand{\ra}{\rangle}
\newcommand{\elem}{\prec}
\newcommand{\uhr}{\restriction}

\newcommand{\power}{\mathcal{P}}

\newcommand{\ol}{\ol}
\newcommand{\po}{\mathbb{P}}
\newcommand{\wo}{\mathbb{W}}
\newcommand{\qo}{\mathbb{Q}}

\newcommand{\co}{\mathbb{C}}

\newcommand{\image}{"} %pointwise image

\newcommand{\F}{\mathcal{F}}

\newcommand{\U}{\mathcal{U}}
\newcommand{\Ult}{\mathrm{Ult}}
\renewcommand{\ol}{\overline}

\renewcommand{\succ}{\suc}
\newcommand{\otp}{\mathop{\mathrm{ot}}}

\DeclareMathOperator{\dom}{dom}
\DeclareMathOperator{\cf}{cf}

\DeclareMathOperator{\suc}{succ}
\DeclareMathOperator{\GCH}{GCH}

\DeclareMathOperator{\cof}{Cof}
\DeclareMathOperator{\col}{Coll}

\DeclareMathOperator{\Sing}{Sing}

\DeclareMathOperator{\HOD}{HOD}

\DeclareMathOperator{\CUB}{CUB}

\def\l{{\langle}}
\def\r{{\rangle}}

\title{On $\omega$-Strongly Measurable Cardinals}
\author{Omer Ben-Neria}
\address{Institute of Mathematics,
 The Hebrew University of Jerusalem,
 Jerusalem 91904, Israel}
\email{omer.bn@mail.huji.ac.il}
\thanks{The first author was partially supported by the Israel Science Foundation Grant
1832/19}

\author{Yair Hayut}
\address{Institute of Mathematics,
 The Hebrew University of Jerusalem,
 Jerusalem 91904, Israel}
\email{yair.hayut@mail.huji.ac.il}
\thanks{
The second author research was partially supported by the Lise Meitner FWF grant 2650-N35 and the Israel Science Foundation Grant 1967/21}
\begin{document}
\begin{abstract}
We prove several consistency results concerning the notion of $\omega$-strongly measurable cardinal in $\HOD$. In particular, we show that is it consistent, relative to a large cardinal hypothesis weaker than $o(\kappa) = \kappa$, that every successor of a regular cardinal is $\omega$-strongly measurable in $\HOD$. 
\end{abstract} 

\maketitle

\noindent

\section{Introduction}

A prominent line of research in set theory is the study of the set theoretic universe $V$ (i.e., model of the axioms of set theory, $\ZFC$) by considering canonical inner model $M \subseteq V$ with additional strong features, which approximates $V$. 
The concept builds on the suggestion that if $M$ is sufficiently ``close'' to $V$, then some of the properties of $M$ may lift to $V$, and allow us to derive new consequences about models of set theory. \\
The prospects of this approach are demonstrated in  the theory of G\"{o}del's constructible universe $L \subseteq V$, and Jensen's Covering Theorem (\cite{JensenCovering}), which asserts that under the anti-large cardinal assumption of the nonexistence of $0^{\#}$, the covering property holds for $L \subseteq V$.\footnote{I.e., every set of ordinals $x \in V$ is contained in a set $y \in L$ such that $|y| \leq |x| + \aleph_1^V$.}
The combination of covering for $L \subseteq V$ together with the rigid structure of $L$, has been shown to have many implications both on cardinal arithmetic in $V$,\footnote{E.g., it implies that the Singular Cardinal Hypothesis ($\mathrm{SCH}$) holds in $V$} as well as on the existence of incompactness phenomena in $V$ such as an Abelian group $G \in V$ of size $\aleph_{\omega+1}$ which is not free, although every subgroup $H \triangleleft G$ of smaller cardinality is free. 
%The main drawback of this result is that it relies on an assumption which prohibits the existence of many large cardinals properties and their appealing consequences. 

Jensen's Covering Lemma describes one side of a sharp dichotomy: If $0^{\#}$ does not exists $L$ covers subsets of $V$ successfully. By Silver, if $0^{\#}$ does exist, $L$ fails to approximate even the most basic features of $V$. For example, in the presence of $0^{\#}$, $L$ never computes successor cardinals correctly.

Jensen's Covering Lemma can be utilized in order to obtain a lower bound for the consistency strength of set theoretical statements, which not necessarily mention large cardinals. But it is quite restrictive: in order to be able to obtain various lower bounds, one has to construct canonical models which can accommodate stronger large cardinal axioms. The construction of such inner models is the subject of a prominent program in set theory, known as the Inner Model program. For a large cardinal property $\Phi$,\footnote{E.g., the existence of a cardinal $\kappa$ with a large cardinal property such as a measurable cardinal, a strong cardinal, a Woodin cardinal, or a supercompact cardinal.} 
 one would like to to construct a canonical $L$-like inner model $K$ which is maximal with respect to inner models which do not satisfy the large cardinal property $\Phi$ (see Schimmerling-Steel \cite{SchimmerlingSteelMaximality} for the precise statement). This maximality property couples with a covering lemma: assuming there is no inner model of $V$ with the property $\Phi$, $K$ approximates the universe $V$ by satisfying a certain covering property. For example, for $\Phi$ being the existence of $0^{\#}$, $K = L$. 

For stronger $\Phi$, the existence of such an inner model $K$ would allow us to extend Jensen's sharp dichotomy. Namely, either the large cardinal property $\Phi$ holds, or $V$ is close to $K$ and therefore inherits various combinatorial properties such as the existence of certain incompactness phenomena.

Starting at around the 1970s, inner models for increasing large cardinal properties $\Phi$ have been constructed. Starting from the seminal studies of Kunen, Silver, and Solovay on a model $L[U]$ with a measurable cardinal (\cite{Kunen1970},\cite{Silver1971}), extended by Dodd-Jensen (\cite{DoddJensen1981}) and Mitchell (\cite{Mitchell1984}), to large cardinal properties $\Phi$ involving coherent sequences of normal measures and many measurable cardinals. Then, following major developments and the introduction of iteration trees in Martin and Steel (\cite{MartinSteel1994}), Mitchell and Steel (\cite{MitchellSteel1994}), and Steel (\cite{Steel1996}), the theory was extended to the level of Woodin cardinals. 

The program took a significant turn after Woodin showed that there cannot be a single maximal inner model, in an absolute sense, past a Woodin cardinal. This has sparked new lines of study, involving forms of the $K^c$ construction (see Jensen, Schimmerling, Schindler, and Steel \cite{Stacking2009} and Andereta, Neeman and Steel \cite{Domestic2001}). 

Another seminal development was the introduction of The Core Model Induction method, first introduced by  Woodin and extensively developed by Steel, Schindler, Sargsyan, Trang and many others (\cite{CMI}). The method establishes new consistency results for stronger large cardinal properties by incorporation ideas form descriptive set theory with various local construction methods. The relevant large cardinal properties are often described in terms of expansions of the Axiom of Determinacy ($\AD$) in inner models $M$ of $\ZF$, and can be further translated to inner models of $\ZFC$ with large cardinal properties. 

First results on fine structural inner models for finite levels of supercompact cardinals were obtained by Neeman and Steel, \cite{neemansteel}, and by Woodin. It is still unknown whether similar constructions could lead to an inner model with a (full) supercompact cardinal, and some recent results of Woodin suggest that major obstructions appear past  the level of finite supercompactness, \cite{Woodin-Midrasha}. There are many excellent resources for the introduction of the inner model theory, the inner model program and its development. %We refer the reader to Jensen (\cite{JensenHistory}), Mitchell (\cite{MitchellHistory}), Neeman (\cite{NeemanLongGames}), Sargysyan (\cite{SargysyanBSL}), Schimerling (\cite{SchimerlingABC}), Steel (\cite{SteelHB}) and Zeman (\cite{Zeman}).
We refer the reader to \cite{JensenHistory, MitchellHistory, NeemanLongGames, SargysyanBSL, SchimerlingABC, SteelHB, Zeman}.

The inner model of Hereditarily Ordinal Definable sets (HOD) plays a significant role in many of the recent advancements in the Inner Model program. 

\begin{definition} 
Let $M$ be a model of set theory. A set $x \in M$ is hereditarily ordinal definable in $M$ if both $x$ and every set in the transitive closure of $x$ is definable in $V$ using some formula with ordinal parameters. The class of all hereditarily ordinal definable sets in a model $M$ is denoted by $\HOD^M$. 

We write $\HOD$ for $\HOD^V \subseteq V$.
\end{definition}

The class $\HOD$ was first introduced by G\"{o}del, and has been extensively studied (for example, see \cite{MyhillScott}). The study of $\HOD$ in inner models of strong forms of $\AD$ and the associated strategic-extender models plays a critical role in Descriptive Inner Model Theory. 
%See Steel and Woodin (\cite{WoodinSteel2016}) and Sargsyan and Trang (\cite{SargsyanTrang2016}, \cite{SargysyanTrangLSA}).
See \cite{SargsyanTrang2016, SargysyanTrangLSA, WoodinSteel2016}. 
%In all known such models a similar dichotomy phenomenon occurs: Either the well-behaved inner model reasonably approximates $V$, or else a certain large cardinal axiom holds, in which case the inner model is very far from $V$.

%The Proper Forcing Axiom is one of is known to be consistent relative to the existence of a supercompact cardinal and it is conjectured that this is the correct large cardinal strength. 

In \cite{Woodin-SuitableExtendersI}, Woodin presents a new approach of addressing inner model problem for all large cardinals.
Woodin analyses the possible properties and limitations of some of the current methods, and introduces the seminal notions of a suitable extender model $N$ for a supercompact cardinal $\delta$ (in $V$), which in addition to several properties similar to well-known inner models, requires that $N$ captures witnessing  $\delta$-supercompact measures in $V$.\footnote{Namely, for all $\lambda > \delta$ there is a supercompact measure $U$ for $P_\delta(\lambda)$ such that $N \cap P_\delta(\lambda) \in U$ and $U \cap N \in N$.}
In a following work (see \cite{Woodin-Midrasha}), Woodin presents the  ``V = Ultimate-L'' axiom, to assert (roughly) that $\Sigma_2$-definable properties of the universe are satisfied in canonical strategic-extender models of the form $\HOD^{L(A,\mathbb{R})} \cap V_{\Theta}$, for some Universally Baire set $A \subseteq \mathbb{R}$.
 
 Combining the above notions, Woodin has formulated the ``Ultimate-L'' conjecture,  asserting that there exists a suitable extender model $N \subseteq \HOD$ %(for a supercompact) 
which satisfies the axiom ``V = Ultimate-L''.
%In this paper, Woodin analyses the possible properties and limitations of hypothetical inner models with large cardinals beyond the current known canonical models, and isolate a few key conjectures towards their constructions. Woodin then focuses on the well studied inner model $\HOD \subseteq V$
%, the Hereditarily Ordinal Definable sets.\footnote{A set $x \in V$ is hereditarily ordinal definable if both $x$ and every set if the transitive closure of $x$ it is definable in $V$ by some formula with ordinal parameters. This model was suggested by G\"{o}del and developed, for example, in \cite{MyhillScott}.} 
Following the search for some $N$, 
the theory established in \cite{Woodin-SuitableExtendersI} studies the possibility of $\HOD \subseteq V$ being a suitable extender model, and possible implication. 
For this, Woodin introduces a new assumption known as the \emph{HOD-conjecture}  (Conjecture \ref{conj-HOD} below) and shows that remarkably, if the $\HOD$-conjecture is true then a sufficiently strong large cardinal assumption (e.g., an extendible cardinal) guarantees a version of the covering lemma for $\HOD \subseteq V$. On the other hand, if the $\HOD$-conjecture fails in the presence of sufficiently large cardinals then $\HOD$ is very far from $V$, just like the smaller inner models. See Theorem \ref{thm:hod-dichotomy} for an exact formulation.

We remark that in general, the inner model $\HOD$ of an arbitrary model of $\mathrm{ZFC}$ can be easily modified by forcing. Nevertheless, it contains every canonical inner model and thus, the $\HOD$-conjecture might be a consequence of the covering theorem for some extremely large canonical inner model, \cite{WoodinUltimateL}. 
An appealing aspect of the $\HOD$-conjecture is that it is a combinatorial statement about the $\HOD$ and $V$, and does not rely on inner model theory. Even without any further development in the inner model program, Woodin established that the $\HOD$-conjecture poses many significant limitations on the consistency of large cardinals in the choice-less context. Moreover, large cardinals beyond choice, if consistent, form a hierarchy of failures of the $\HOD$-conjecture, \cite{BKW}.

The $\HOD$-conjecture centers around the notion of $\omega$-strongly measurable cardinals in $\HOD$. Not much was know about this notion, and previously, Woodin has raised the question (\cite{Woodin-Midrasha}) of whether more than three $\omega$-strongly measurables in $\HOD$ can exist. In this work, we study the notion of $\omega$-strongly measurable cardinals in $\HOD$, we prove several consistency results concerning this notion, and establish the consistency of a model where all successors of regular cardinal are $\omega$-strongly measurable in $\HOD$. 

%We proceed to introduce the central notions. 

\begin{definition}\label{def:omega.strongly.measurable}
 Let $\kappa$ be an uncountable regular cardinal, and $S$ a stationary subset of $\kappa$. 
 We say that $\kappa$ is strongly measurable in $\HOD$  with respect to $S$ if there exists some $\eta < \kappa$ such that
  $(2^\eta)^{\HOD} < \kappa$ and there is no partition $\l S_\alpha \mid \alpha < \eta\r \in \HOD$ of $S$ into sets, all stationary sets in $V$.
 We say that $\kappa$ is $\omega$-strongly measurable in $\HOD$ if it is strongly measurable in $\HOD$ with respect to the set  $S = \kappa \cap \cof(\omega)$, and that it is strongly measurable in $\HOD$ if it is strongly measurable in $\HOD$ with respect to $S = \kappa$. 
\end{definition}
In general, one might replace $\HOD$ with any other inner model of $V$, $M$, obtaining a meaningful notion of strong measurability in $M$. Since in this paper we will be interested solely in strong measurability in $\HOD$, we will occasionally omit the emphasis ``in $\HOD$'' and say simply that $\kappa$ is strongly measurable.  

It is shown in \cite{Woodin-SuitableExtendersI} that if $\kappa$ is an $\omega$-strongly measurable in $\HOD$ then there are stationary sets $S \subseteq \kappa \cap \cof(\omega)$ for which the restriction of the filter $\CUB_\kappa \uhr S$ to $\HOD$, forms a measure on $\kappa$ in $\HOD$. 
On the other hand, Woodin shows (\cite{Woodin-SuitableExtendersI}) that the existence of a class of regular cardinals which are not $\omega$-strongly measurable in $\HOD$, together with the existence of a $\HOD$-supercompact, implies that $\HOD$ satisfies many appealing approximation properties with respect to $V$. The results promote Woodin's $\HOD$-conjecture.

\begin{conjecture}[$\HOD$ conjecture, {\cite[Definition 191]{Woodin-SuitableExtendersI}}]\label{conj-HOD}
There is a proper class of regular uncountable cardinals $\kappa$ which are not
$\omega$-strongly measurable in $\HOD$.
\end{conjecture}

In light of the $\HOD$-conjecture, it is natural to attempt forming models with as many as possible $\omega$-strongly measurable cardinals in $\HOD$. 
Woodin has established the consistency (relative to large cardinals) of models with up to three $\omega$-strong measurable cardinals (see \cite[Remark 3.43]{Woodin-Midrasha}). 
The main purpose of this work is to prove that many strongly measurable cardinals can be obtained from relatively mild large cardinal assumption of hyper-measurability. 

\begin{theorem}\label{THM1}
It is consistent relative to the existence an inaccessible cardinal $\theta$ for which $\{ o(\kappa) \mid \kappa < \theta\}$ is unbounded in $\theta$,
that every successor of a regular cardinal is strongly measurable in $\HOD$.\end{theorem}

Cummings, Friedman, and Golshani (\cite{CFG}) have established the consistency of a model where $(\alpha^+)^{\HOD} < \alpha^+$ for every infinite cardinal $\alpha$.
In \cite[Theorem 2.2]{GitikMerimovich}, Gitik and Merimovich prove that it is consistent relative to large cardinals that every regular uncountable cardinal is measurable in $\HOD$. A similar result is obtained using a different technique in \cite[Theorem 1.4]{bunger}. Perhaps, more related to our work is \cite[Theorem 1.3]{bunger}, in which a club of cardinals which are measurable in $\HOD$ is obtained from a large cardinal axiom weaker than $o(\kappa) = \kappa$. 
In those models there are no $\omega$-strongly measurable successor cardinals. 

We note that these results do not apply to models where there is an extendible cardinal. The existence of an extendible cardinal in $V$ derives a sharp dichotomy between $\HOD$ being either very close or very far from $V$, as shown by Woodin's $\HOD$-Dichotomy Theorem (\cite{Woodin-SuitableExtendersI}).

\begin{theorem}[The $\HOD$-Dichotomy, Woodin, {\cite{WJD-HOD-Dichotomy-Survey}}]\label{thm:hod-dichotomy}
Let $\delta$ be an extendible cardinal. Then one of the following holds:
\begin{enumerate}
\item Every cardinal $\eta$ above $\delta$ which is singular in $V$, is singular in $\HOD$ and 
$(\eta^+)^{\HOD} = \eta^+$.
\item Every regular cardinal above $\delta$ is $\omega$-strongly measurable in $\HOD$.
\end{enumerate}
\end{theorem}

\vskip\medskipamount

In the last part of this work, we prove a consistency result regarding strong measurability at successors of singular cardinals. 
Woodin (\cite{Woodin-SuitableExtendersI}) establishes the consistency of a successor of a singular cardinal $\lambda$, which is $\omega$-strong measurable cardinal in $\HOD$, from the large cardinal assumption $I_0$.
Here, we prove a weaker consistency result from a weaker large cardinal assumption. 

\begin{theorem}\label{thm:singular.strong.meas}
Suppose that $\kappa < \lambda$ are cardinals, $\kappa$ is $\lambda$-supercompact and $\lambda$ is measurable. Then, there is a generic extension in which $\kappa$ is a singular cardinal of cofinality $\omega$, and  $\lambda = \kappa^+$  is
strongly measurable in $\HOD$ with respect to $S$, for some stationary subset 
$S \subseteq \lambda \cap \cof(\omega)$.
\end{theorem}
\vskip\bigskipamount
\emph{A brief summary of this paper}. In \textbf{section \ref{section:strong measurability}} we review some basic facts about strong measurability which will be central in the proof of the main theorem.  
In the following sections, we gradually develop the forcing methods used to prove our main results (Theorems \ref{THM1} and \ref{thm:singular.strong.meas}): 
In \textbf{section \ref{Sec-omega1}} we show how to obtain a model where $\omega_1$ is strongly measurable in $\HOD$ starting with a single measurable cardinal. The case of $\kappa = \omega_1$ is different from the general case as it does not require incorporating posets for changing cofinalities. It can also be seen as a warm-up for the general case. 
In \textbf{section \ref{Sec-Cforcing}} we further develop the ideas from the previous section and combine them with a suitable iteration for changing cofinalities. As a result, we establish the consistency of a strongly measurable cardinal which is a successor of an arbitrary regular cardinal $\lambda$, from the large cardinal assumption of $o(\kappa) = \lambda+1$. 
In \textbf{section \ref{Sec-PrikryEmbedding}} we introduce a method to construct a Prikry-type poset which is equivalent to the forcing from the previous section, and has a direct extension order that is $\lambda$-closed. This is utilized in 
\textbf{section \ref{Sec-Many}} to form iterations of the single cardinal forcing, thus obtaining models with many strongly measurable cardinals.
In \textbf{section \ref{Sec:suc.of.sing}} we prove our theorem concerning successors of singular cardinals. The results of this section do not depend on the other sections past our preliminaries. 

In the appendix we cite and prove 
some useful results related to
homogeneous forcings and their iterations (including Prikry type forcings), 
and homogeneous iterations for changing cofinalities.

Our notations are mostly standard. We follow the Jerusalem forcing convention in which for two conditions $p,p'$ in a poset $\po$, the fact $p'$ is stronger (more informative) than $p$ is denoted by $p' \geq p$. 

\section{Variations of Strong Measurability}\label{section:strong measurability}
We start with several observations concerning a natural generalization of the notion of $\omega$-strongly measurability. 
%First, we would like to give a notation for a local version of strong measurability.
\begin{definition}
Let $S \subseteq \kappa$, $S \in \HOD$ stationary and $\eta$ a cardinal in $\HOD$. $\kappa$ is $(S,<\eta)$-strongly measurable if there is no partition in $\HOD$ of $S$ into $\eta$ many disjoint stationary sets. $\kappa$ is $(S,\eta)$-strongly measurable if it is $(S, <(\eta^{+})^{\HOD})$-strongly measurable. 
\end{definition}
\begin{definition}
A cardinal $\kappa$ is $S$-strongly measurable if $S \in \HOD$ and $\kappa$ is $(S, <\eta)$-strongly measurable for some $\eta$ such that $(2^\eta)^{\HOD} < \kappa$. We say that $\kappa$ is strongly measurable if $\kappa$ is $\kappa$-strongly measurable.
\end{definition}
Therefore a cardinal $\kappa$ is $\omega$-strongly measurable if it is $(S^{\kappa}_{\omega},\eta)$-strongly measurable for $\eta$ such that $\left(2^\eta\right)^{\HOD} < \kappa$. Note that if $S \subseteq T$ are stationary subsets of $\kappa$ in $\HOD$ and $\kappa$ is $T$-strongly measurable then it is $S$-strongly measurable. In particular, every strongly measurable cardinal is $\omega$-strongly measurable.

\begin{theorem}[Woodin]
Let $\delta$ be an extendible cardinal.
Then the following are equivalent:
\begin{enumerate}
\item There is a regular cardinal $\kappa \geq \delta$ which is not $\omega$-strongly measurable.
\item There is a regular cardinal $\kappa \geq \delta$ which is not $(S, \delta)$-strongly measurable for some $S \in \HOD$ which consists of singular ordinals of cofinality $<\delta$. 
\item The $\HOD$-conjecture.
\item There is no regular $\omega$-strongly measurable cardinal above $\delta$. 
\end{enumerate}
\end{theorem}
For the proof see \cite[Theorems 197, 212, 213]{Woodin-SuitableExtendersI}.
Without the assumption that the ordinals of $S$ have fixed $V$-cofinality the equivalence might fail. 

The next result provides a necessary and sufficient condition for a cardinal $\kappa$ to be $\omega$-strongly measurable in $\HOD$. 
This observation will guide us in devising the main forcing construction, which will be used to prove theorem \ref{THM1}.

\begin{lemma}\label{lem:sufficient.omega.strongly}
A cardinal $\kappa$ is strongly measurable with respect to $S \in \HOD$ if and only if
$\kappa$ is an inaccessible cardinal in $\HOD$ and the restriction of the club filter on $S$ to $\HOD$ is the intersection of  $\eta$ normal measures from $\HOD$, $\l U_{\kappa,i} \mid i < \eta\r \in \HOD$, for some $\eta < \kappa$. %Then $\kappa$ is $\left(\kappa, \eta\right)$-strongly measurable.
\end{lemma}
\begin{proof}
For the backwards implication, since $\kappa$ is inaccessible in $\HOD$ and $\eta < \kappa$,  $\left(2^\eta\right)^{\HOD} < \kappa$.

Let $\langle T_\alpha \mid \alpha < (\eta^+)^{\HOD}\rangle\in \HOD$ be a decomposition of $S$ into stationary sets. By the assumption, for each $\alpha$ there is a measure $U_{\kappa,i}$ in $\HOD$ such that $T_\alpha \in U_{\kappa,i}$. Since the sets $T_\alpha$ are pairwise disjoint, it is impossible for $\alpha \neq \beta$ to belong to the same $U_{\kappa,i}$. Thus, we obtain an injective function from $(\eta^+)^{\HOD}$ to $\eta$ in $\HOD$---a contradiction.

Let us assume now that $\kappa$ is strongly measurable with respect to $S \in \HOD$. In particular, $\kappa$ is inaccessible in $\HOD$. Let $\mathcal{S} \in \HOD$ be a maximal collection of pairwise disjoint stationary subsets of $S$, in $\HOD$, such that for all $T \in \mathcal S$, the club filter restricted to $T$ is an ultrafilter in $\HOD$. Let us denote this ultrafilter by $U_T$. Since this collection is a partition of $S$ into stationary sets, $|\mathcal S| < \kappa$.

If $\bigcap \{U_T \mid T \in \mathcal S\}$ is not the club filter restricted to $S$ in $\HOD$, then it contains a set $S \setminus S'$, where $S' \subseteq S$ stationary, $S' \in \HOD$. In particular, $S' \notin U_T$ for all $T \in \mathcal S$, so $S' \cap T$ is non-stationary for all $T \in \mathcal S$, and $S'' = S' \setminus \bigcup_{T \in \mathcal S} (S' \cap T)$ is a stationary subset of $S$, disjoint from all members of $\mathcal S$. Since $\kappa$ is strongly measurable with respect to $S$, it is also strongly measurable with respect to $S''$, and thus there is some $T' \subseteq S''$ stationary such that the club filter restricted to $T'$ is an ultrafilter. But this contradicts the maximality of $\mathcal{S}$.    
\end{proof}

\begin{corollary}\label{lem:necessary.omega.strongly}
Let $\kappa$ be $(S,{<}\kappa)$-strongly measurable. Then $S$ is contained in the regular cardinals of $\HOD$, up to a non-stationary error. 
\end{corollary}
\section{$\omega_1$ is strongly measurable from one measurable cardinal}\label{Sec-omega1}
In this section, we would like to present a forcing that forces $\omega_1$ to be strongly measurable. By Lemma \ref{lem:sufficient.omega.strongly}, this means that in $\HOD$, the club filter of $\omega_1$ is an intersection of countably many normal measures. In the case of $\omega_1$, we can take a single measure. So, we would like to collapse a measurable cardinal $\kappa$ with a normal measure to be $\omega_1$ and then using a Mathias-type forcing, to add a club that diagonalizes the normal measure. In order to show that this works, we need to show two things. First, we must show that the iteration is cone homogeneous. This is done in Lemma \ref{Lem0:CHomog}. Second, we need to show that it does not collapse $\omega_1$. This amounts to show that the second step of the iteration is $\sigma$-distributive, which in turn requires us to be able to add a $U$-generic point to the generic club, see Lemma \ref{Lem01} for the precise formulation.  

Let us present the forcing. Suppose that $\kappa$ is a measurable cardinal in a model $V$ and $U$ is a normal measure on $\kappa$. 
Force with Levy collapse poset $\col(\omega,<\kappa)$ over $V$. Let $H$ be a $V$-generic filter. 
%\begin{definition}
 
 Working in the generic extension $V[H]$, let $\co_U$ be the poset consisting of pairs $x = \l c, A\r$, where $c \subseteq \kappa$ is a bounded closed subset of $\kappa$ and $A \in U$. The condition $x' = \l c',A'\r$ extends $x$ if $c'$ is an end extension of $c$, $A' \subseteq A$, and $c'\setminus c \subseteq A$. 
%\end{definition}

It is clear that if $x = \l c,A\r$ and $x' = \l c',A'\r$ are two conditions with the same bounded closed set $c = c'$ then $x,x'$ are compatible. Since
$\kappa^{<\kappa} = \kappa$ in $V[H]$ then $\co_U$ satisfies $\kappa^+$-c.c.\ (which is $\aleph_2^{V[H]}$-c.c.).
The forcing $\co_U$ adds a diagonalizing club to $U$. It has also been studied in \cite{Rinot2009} in the context of well-behaved posets which can introduce square sequences, and was found useful in other contexts. 

The following lemma is the key ingredient in the proof of the distributivity of $\co_U$.
\begin{lemma}\label{Lem01}
 Work in $V[H]$ and fix some regular cardinal $\theta > \kappa^+$. There exists a stationary set of structures $M \elem H_\theta$ of size $|M| < \kappa$, with the property that $\sup(M \cap \kappa) \in A$ for every $A \in U \cap M$.
\end{lemma}
\begin{proof}
 Fix any $f\colon H_\theta^{<\omega} \to H_\theta$ in $V[H]$. We would like to show that there exists some $M \subseteq H_\theta$ which is closed under $f$ and satisfies the conditions in the statement of the lemma. 

 Fix in $V$ a name $\name{f}$ for $f$ and let $f' \colon\col(\omega, <\kappa) \times (H_\theta^V)^{<\omega} \to H_{\theta}^V$ be a function that sends $(p, x) \in \col(\omega,<\kappa)\times (H_\theta)^{<\omega}$ to $y$ if $p \Vdash \name{f}(x) = \check{y}$. Note that $x$ is a finite sequence of names.
 
 By the definition of $f'$, if $M' \prec H_\theta^V$ 
 is closed under $f'$, $M' \cap \kappa \in \kappa$, and $\kappa, U \in M'$,  
 then $M' = M \cap H_{\theta}^V$ for some 
 $M \subseteq H_\theta$ which is closed under $f$.
 Indeed, we may take $M = M'[H \cap M]$. By the chain condition of $\col(\omega,<\kappa)$,
 every name for a ground model object that belongs to $M'$, can be refined to
 a nice name which is contained in $M'$.
 
 Since $U \subseteq H_{\theta}^V$, it is therefore sufficient to prove that there exists some $M' \subseteq H_\theta^V$ which is closed under $f'$ and satisfies
 $|M'| < \kappa$ and $\sup(M' \cap \kappa) \in A$ for all $A \in M' \cap U$. 
 
 %Let $\name{f'}$ be a $\col(\omega,<\kappa)$ name for $f'$. 
 %Since $\col(\omega,<\kappa)$ is $\kappa$-c.c. there exists a function $F'$ in $V$, $F' \colon (H_\theta^V)^{<\omega} \to (H_\theta^V)^{<\kappa}$ such that
 %$0_{\col(\omega,<\kappa)} \Vdash \name{f'}(\check{x}) \in \check{F'(x)}$ for all $x \in (H_\theta^V)^{<\omega}$.
 Working in $V$, take an elementary substructure $N \elem H_\theta^V$ satisfying $f'[N] \subseteq N$,  $N^{<\kappa} \subseteq N$, $|N| = \kappa$, $\kappa \in N$. 
 Let $j \colon V \to W \cong \Ult(V,U)$ be the ultrapower embedding induced by $U$. 
 Consider the structure $\tilde{M}' = j\image N \elem j(H_\theta^V)$. $\tilde{M}' \in W$, is closed under $j(f')$, and $\tilde{M}' \cap j(\kappa) \in j(\kappa)$.   
 It follows that $0_{j(\col(\omega,<\kappa))}$ forces $\tilde{M}'$ to be closed under $j(\name{f})$.
 Finally, %$\tilde{M}' \cap j(\kappa) = \kappa$ and
 for every $A \in \tilde{M}' \cap j(U)$, $A = j(\bar{A})$ for some $\bar{A}\in U$ and therefore $\kappa \in A$. 
 
 So $\tilde{M}$ satisfies the conclusion of the lemma in $W$, and it is closed under $j(f')$. By the elementarity of $j$, there is $M' \in V$ satisfying the conclusion of the lemma, and closed under $f'$, 
 and thus $M'[H \cap M'] = M$ satisfies the requirements of the lemma.
\end{proof}
%\begin{remark}
%As this level, the argument is essentially a diagonal intersection argument (or preservation of normality): for each $\kappa$-sequence of measure one sets in $V[H]$, their diagonal intersection is also measure one, and thus there is $\alpha$ such that $\alpha$ is forced to appear in all measure one sets which are enumerated before $\alpha$. The normality argument does not seem to generalize to the more complicated situation in the next section. 
%\end{remark}

\begin{proposition}\label{prop01}
 $\co_{U}$ is $\kappa$-distributive.
\end{proposition}
\begin{proof}
 Since $\kappa = \aleph_1$ in $V[H]$, we need to check that the intersection of a countable family $\{ D_n \mid n<\omega\}$ of dense open subsets of $\co_{U}$ is dense.  Pick some regular cardinal $\theta  > \kappa^+$ such that $\co_U, \{D_n \mid n<\omega\} \in H_\theta$. 
 By lemma \ref{Lem01}, for every condition $x \in \co_{U}$ there exists an elementary substructure $M \elem H_\theta$ of size $|M| = \aleph_0$, with $x, \po,\co_{U},\{D_n \mid n < \omega\} \in M$ and further satisfies that $\sup(M \cap \kappa) \in A$ for every $A \in M \cap \kappa$. We may also assume that $M = M'[H \cap M']$ for $M' \in V$, $M' \prec H_\theta^V$.

 Denote $\sup(M \cap \kappa)$ by $\alpha$ and pick a cofinal sequence $\l \alpha_n \mid n < \omega\r$ in $\alpha$. 
 We can construct  an increasing sequence of extensions $\l x_n \mid n<\omega\r \subseteq M$ of $x$, $x_n = \l c_n,A_n\r$ such that $x_{n+1} \in D_n$ and $\max(c_n) \geq \alpha_n$ for every $n<\omega$. Without loss of generality, we may assume that for every $A \in U \cap M$ there is $n < \omega$ such that $A_n \subseteq A$.
 
 Since $x_n = \l c_n,A_n\r \in M$ then $\alpha \in A_n$ for all $n < \omega$. 
 It follows that $x^* = \l \{ \alpha \} \cup (\bigcup_n c_n), \bigcap_n A_n\r$ is a condition in $\co_U$, which is clearly an upper bound of $\l x_n \mid n < \omega\r$.\footnote{Note that $\bigcap A_n = \bigcap_{A \in M' \cap U} A$.} 
 We conclude that there exists $x^*$ extending our given condition $x$ such that $x^* \in \bigcap_n D_n$. 
\end{proof}

\begin{lemma}\label{Lem0:CHomog}
 $\co_{U}$ is cone homogeneous.
\end{lemma}
\begin{proof}
 Let $x_1 = \l c_1,A_1\r$, $x_2 = \l c_2,A_2\r$ be two conditions of $\co_{U}$. Take $\nu \in A_1 \cap A_2$ above $\max(c_1),\max(c_2)$
 and consider the extensions $y_1 = \l c_1 \cup \{\nu\}, (A_1 \cap A_2) \setminus (\nu+1)\r$, $y_2 =  \la c_2 \cup \{\nu\}, (A_1 \cap A_2) \setminus (\nu+1)\r$ of $x_1$ Define a cone isomorphism $\sigma\colon \co_U/y_1 \to \co_U/y_2$ by
 \[\sigma( \l c, A\r) = \l c_2 \cup (c\setminus \nu), A\r \]
 $\sigma$ is clearly order preserving map onto $\co_U/y_1$,  and has an order preserving inverse which is given by 
 \[    
   \sigma^{-1}( \l c, A\r) = \l c_1 \cup (c\setminus \nu), A\r
 \]
 \end{proof}

\begin{theorem}\label{Thm-0stmeas}
Suppose $C \subseteq \co_U$ is a $V[H]$-generic filter. Then, in $V[H \ast C]$,  
 $\kappa = \aleph_1^{V[H \ast C]}$ is strongly measurable. 
\end{theorem}
\begin{proof}
By Lemma \ref{corr: OD implies homogeneous}, $\col(\omega,<\kappa) * \co_U$ is cone homogeneous, and therefore $\HOD^{V[H \ast C]} \subseteq V$. 
 It is clear from the definition of $\co_U$ that for every subset $S \subseteq \kappa$ in $V$, $S$ is stationary in $V[H \ast C]$ if and only if 
 $S \in U$. It follows that the closed unbounded filter on $\kappa = \aleph_1^{V[H \ast C]}$ in $V[H \ast C]$ is a $\HOD$-ultrafilter. Therefore $V[H \ast C]\models \kappa$ is strongly measurable. 
\end{proof}

%\newpage
\section{Strongly measurable successor of a regular cardinal}\label{Sec-Cforcing}
In this section we would like to force a successor of an uncountable regular cardinal, $\kappa = \lambda^+$ to be strongly measurable. There are a few difficulties that arise. First, there is a definable splitting of the ordinals below $\kappa$, according to the cofinalities, so the club filter cannot be an ultrafilter but rather an intersection of a few normal measures. This means that we should fix a collection of normal measures, that their intersection is indented to become the club filter. Moreover, when killing a stationary set which is small with respect to the designated filter, we are forcing a club through the previous regulars, which are now going to change cofinalities to various possibilities. This means that a Levy collapse by itself would not provide all the cofinality changes that we need, and we must use a more complicated method of changing cofinalities in a homogeneous way. 

Suppose that $\lambda < \kappa$ are two cardinals such that $\lambda$ is regular and $\kappa$ is measurable with $o(\kappa) = \lambda+1$. Let $\mathcal{U} = \l U_{\alpha,\tau} \mid \lambda < \alpha \leq \kappa, \tau < o^{\mathcal{U}}(\alpha)\r$ be a coherent sequence of normal measures with $o^{\mathcal{U}}(\kappa) = \lambda + 1$.

Let $\po^\mathcal{U}_\kappa = \l \po_\alpha, \qo_\alpha \mid \alpha < \kappa\r$ be the homogeneous iteration of subsection \ref{ssec-non-stationary-iteration}. In the next two sections, $\po$ stands for $\po^\mathcal{U}_\kappa$. For the main properties of $\po$, we refer the reader to Fact \ref{fact:summary-properties-of-PU}. We will explicitly need the following additional property of the iteration.

\begin{remark}\label{rmk1}
We note that it follows at once from the definition of $\qo_\alpha$ that every $V$-set $A \in \bigcap_{i<o^{\U}(\alpha)} U_{\alpha,i}$ contains a tail of the cofinal sequence $b_\alpha$. This is because 
every condition $q = \l t, T\r \in \qo_\alpha$ has a direct extension $q^A = \l t, T^A\r$ of $q$, which satisfies that $\suc_T(s) \subseteq A$ for all $s \in T$. 
\end{remark}

\begin{definition}
Let $G \subseteq \po$ be a $V$-generic filter and $H \subseteq \col(\lambda,<\kappa)$ be the Levy collapse generic over $V[G]$. 
Working in $V[G \ast H]$ we consider the filter $\F_\kappa$ generated by $\bigcap_{i \leq \lambda} U_{\kappa,i}$, 
\[ \F_\kappa = \{ A \subseteq \kappa \mid \exists B \in \bigcap_{i\leq \lambda} U_{\kappa,i}, B \subseteq A\}. \]
\end{definition}
%Equivalently, $A \in \F_\kappa$ if and only if there is $B \in \bigcap_{i \leq \lambda} U_i$ such that $B \subseteq^* A$.
The filter $\mathcal{F}_\kappa$ is going to generate the club filter in $\HOD$ in the generic extension.  
\begin{lemma}\label{Lem1}
 $\F_\kappa$ is a $\kappa$-complete filter in $V[G\ast H]$. 
\end{lemma}
\begin{proof}
Suppose that $\l A_\nu \mid \nu < \beta\r \in V[G\ast H]$ is a sequence of $\beta < \kappa$ many sets of $\F_\kappa$. We would like to show that $\bigcap_{\nu < \beta} A_\nu$ belongs to $\F_\kappa$. We may assume that $A_\nu \in \bigcap_{i \leq \lambda} U_{\kappa_i}$ for all $\nu < \beta$. 

In order to prove the claim, we move from $V[G\ast H]$ to $V[G]$, and then to $V$. 
Working in $V[G]$, let $\name{A_\nu}$ be a $\col(\lambda,<\kappa)$-name for the $V$-set in $\bigcap_{i\leq\lambda} U_{\kappa,i}$. Since $\col(\lambda,<\kappa)$ satisfies $\kappa$-c.c., there exists a family of $V$-sets $X_\nu  \subseteq \bigcap_{i\leq\lambda} U_{\kappa,i}$, $X_\nu \in V[G]$, of size $<\kappa$ such that $\Vdash\name{A_\nu} \in \check{X}_\nu$. Fix in $V$ a $\po$-name $\name{X}_\nu$ for each $X_\nu$.
 %Let $\l A_{\nu,j} \mid j < \beta_\nu\r$ be an enumeration of $X_\nu$ in $V[G]$, in length $\beta_\nu < \kappa$. 

 Let $p \in G$ be a condition forcing the above. Moving back to $V$, the fusion lemma for non-stationary support iteration of Prikry type forcings, \cite[Lemma 3.6]{bunger}, guarantees
 that there exists some $q \in G$ and a sequence of sets $\l Y_\nu \mid \nu < \beta\r$ in $V$, 
 so that for each $\nu < \beta$, $Y_\nu \subseteq \bigcap_{i \leq \lambda}U_{\kappa,i}$ has size $|Y_\nu| < \kappa$, 
 and $q\Vdash \name{X}_\nu \subseteq \check{Y}_\nu$. 
 For each $\nu$, let $A'_\nu = \bigcap Y_\nu$, and $A' = \bigcap_{\nu < \beta} A'_\nu$. 
 Since $U_{\kappa,i}$ is $\kappa$-complete for all $i \leq \lambda$, we have in $V[G\ast H]$ that
 $A' \in \F_\kappa$ and $A' \subseteq \bigcap_{\nu < \beta} A_\nu$. 
\end{proof}
%\begin{remark}
%A similar argument shows that $\F_\kappa$ is normal.
%\end{remark}
To produce a model where $\kappa$ is $\omega$-strongly measurable, we will force over $V[G\ast H]$ to add 
a closed unbounded set $C \subseteq \kappa$ which is almost contained in every set $A \in \F_\kappa$. 

\begin{definition}
 Working in a $V$-generic extension $V[G\ast H]$ by $G\ast H \subseteq  \po*\col(\lambda,<\kappa)$, 
 we define the forcing $\co_{\F_\kappa}$.  Conditions $x \in \co_{\F_\kappa}$ are pairs $x = \l c, A\r$ where
 $c$  is a closed and bounded subset of $\kappa$ and $A \in \F_\kappa$. 
 A condition $x' = \l c', A'\r \in \co_{\F_\kappa}$ extends $x$ (denoted $x' \geq x$) if
 \begin{itemize}
 
 \item[(i)] $c' \cap \max(c) = c$, 
 \item[(ii)] $A' \subseteq A$, and 
 \item[(iii)] $c' \setminus c \subseteq A$.  
 \end{itemize}

\end{definition}

For conditions $x = \l c,A\r \in \co_{\F_{\kappa}}$ we will frequently denote $c$ and $A$ by  $c^x$ and $A^x$ respectively. 
It is clear that if $R \subseteq \co_{\F_\kappa}$ is generic then the union $C = \bigcup\{ c^x \mid x \in R\}$ is a closed and unbounded subset of $\kappa$ which is almost contained in every $A \in \F_{\kappa}$. 
Since $\F_\kappa$ is a filter and $\kappa^{<\kappa} = \kappa$, the forcing $\co_{\F_\kappa}$ is $\kappa$-centered and therefore satisfies $\kappa^+$-chain condition.

The following lemma is a parallel of Lemma \ref{Lem01}. From this lemma we will infer the distributivity of the forcing $\co_{\F_\kappa}$.

\begin{lemma}\label{Lem2}
 Working in $V[G\ast H]$, for any regular cardinal $\theta > \kappa^+$ and $\tau \leq \lambda$, there exists a stationary set of structures $M \elem H_\theta$ with $\sup(M \cap \kappa) = \alpha$ which satisfy 
\begin{itemize} 
 \item[(i)] $M^{<\tau} \subseteq M$; 
 \item[(ii)] $o^{\U}(\alpha) = \tau$; 
 \item[(iii)] For every $A \in \F_\kappa \cap M$, $\alpha \in A$ and moreover $b_\alpha \subseteq^* A$ (namely $b_\alpha \setminus A$ is bounded in $\alpha$).
\end{itemize}
 
\end{lemma}
\begin{proof}
Fix a function $f\colon [H_{\theta}]^{<\omega} \to H_{\theta}$ in $V[G\ast H]$. Back in the ground model $V$, let $\name{f}$ be a $\po * \col(\lambda,<\kappa)$-name for $f$. Since $\col(\lambda,<\kappa)$ is $\kappa$-c.c., there exists a $\po$-name function $\name{F} \colon [H_\theta^V]^{<\omega} \to [H_\theta^V]^{<\kappa}$ such that $\name{f}(x)$ is forced to be a member of $\name{F}(x)$ for every $x \in H_\theta^V$.

Let us consider our ability to approximate $\name{F}$ in $V$. Let $N \elem H_\theta^V$ be an elementary substructure of size $\kappa$ with $N^{<\kappa} \subseteq N$ and $\kappa, \po,\name{F}\in N$.

\begin{claim}\label{claim:properness}
Let $N$ be as above and $p \in \po \cap N$. Then, there is $p^* \leq p$ which is \emph{$N$-generic}, namely for every name for an ordinal $\name\sigma \in N$, there is set of ordinals $S \in N$ such that $S \subseteq N$ and $p^* \Vdash \name\sigma \in S$.
\end{claim} 
\begin{proof}
 By a standard argument concerning capturing dense open sets in Prikry-type forcings and fat-trees (e.g., see \cite{Gitik-HB}) %\todo{Put exact reference. Also isn't it just the fusion lemma from before?} 
 for every dense open set $D$ of $\po$, $p \in \po$ there exists a direct extension $p' \geq^* p$ which reduces capturing $D$ to a dense subset of $\po_\mu$ for some $\mu < \kappa$, namely the set of all $r \in \po_\mu$ such that $r ^\smallfrown p' \restriction [\mu,\kappa) \in D$ is dense below $p' \restriction \mu$. 
 Moreover, given $\nu < \kappa$ we can also make the direct extension $p'$ to agree with $p$ up to $\nu+1$ (i.e., $p'\uhr\nu+1 =  p\uhr\nu+1$) in which case $\mu > \nu$. 

 Given an initial condition $p \in \po$, we can list the dense open sets in $N$, 
 $\l D_i \mid i < \kappa\r$, and form an increasing sequence of direct extensions of $p$, $\l p^i \mid i < \kappa\r$, together with a closed unbounded set $C^* = \l \nu_i \mid i  <\kappa\r$ such that for every successor ordinal $i = i'+1$,  $p^i \in N$ reduces the dense set $D_{i'}$ of $\po$ to a bounded dense set $D_{i'}'$ of $\po_{\mu_i}$ for some $\nu_i <\mu_i < \kappa$, and $p^i\uhr \nu_i+1 = p^{i'}\uhr \nu_i+1$. 
 By a standard argument concerning non-stationary support iterations (e.g., see the fusion argument in the proof of \cite[Lemma 2.2]{bunger}), the sequence of direct extensions $\l p^i \mid i < \kappa\r$ has an upper bound $p^* \geq^* p$. 
 It follows that for every $\po$-name $\name{\sigma} \in N$ of an element of $H_\theta^V$, there exists some $\mu < \kappa$ and 
 a $\po_\mu$-name $\name{\sigma}' \in N$ such that $p^* \Vdash \name{\sigma} = \name{\sigma}'$. %We call a condition $p^*$ satisfying this property \emph{$N$-generic}.

 In particular, for 
 each such name $\name{\sigma}$, $p^*$ forces that it can take $<\kappa$ many values in $H_\theta^V$, all of which are in $N$. This follows from the elementarity of $N$ in $H_\theta$ and the fact $\kappa+1 \subseteq N$. 
  \end{proof}

 Let $j_\tau\colon V \to M_\tau$ be the ultrapower embedding by $U_{\kappa,\tau}$ and  $M' = j_{\tau}\image N \elem j_\tau(H_\theta^V)$, $M' \in M_\tau$. 
\begin{claim}\label{claim:closure of j_tau''N}
 $j_\tau(p^*)$ forces that $M'$ is closed under $j(\name{F})$. 
 \end{claim}
 \begin{proof}
 Indeed 
 if $G^* \subseteq j_\tau(\po)$ is $M_\tau$-generic with $j_\tau(p^*) \in G^*$ 
 then for each 
 $\mu < \kappa$, $G^*_\mu = \{ p\uhr \mu \mid p \in G^*\}$ is a $V$-generic filter for $\po_\mu$. 
 For every  $y = j_\tau(\name{\sigma})_{G^*} \in M' \cap j_\tau(H_\theta^V)$ and $F^* = j_\tau(\name{F})_{G^*}$, $F^*(y) = j_\tau(\name{F})\left( j_\tau(\name{\sigma}) \right)_{G^*}$ is the $G^*$-generic interpretation of the $j_\tau(\po)$-name $j_\tau\left( \name{F}(\name{\sigma}) \right)$. As $p^*$ forces $\name{F}(\name{\sigma}) = \name{\sigma}'$ for some $\name{\sigma}' \in N$ 
 which is a $\po_\mu$-name for some $\mu < \kappa$, 
 we see that $j_\tau(p^*)$ forces $j_\tau\left( \name{F}(\name{\sigma}) \right)= j_\tau(\name{\sigma}')$, where $j_\tau(\name\sigma')$ is a $j_\tau(\po_\mu) = \po_\mu$-name.
 If $q \in G_\mu$ and $z \in N \cap H_\theta^V$ are such that $q \Vdash_{\po_\mu} \name{\sigma}' = \check{z}$ then 
 $j_\tau(q) = q \Vdash j_\tau(\name{\sigma}') = j_\tau(\check{z})$. We conclude that 
 $F^*(y) = z \in M'$.
 \end{proof}
 
 We now return to prove the statement of the lemma. It
  is sufficient to prove that in $V[G]$ there exists some 
  $M' \subseteq H_\theta^V$ which is closed under $F$ and satisfies
 requirements (i)-(iii).
 Let $p \in \po$ be a condition. 
 By a standard density argument there are $N \elem H_\theta^V$ and $p^* \in G$ which is $N$-generic, with $p^* \geq^* p$.\footnote{This is true, since for every $q \geq p$ there is $r \geq^* p$ such that $q \geq r$ is a finite Prikry extension. Let $q^* \geq^* q$ be $N$-generic. Then, there is $\alpha$ such that $q \restriction [\alpha, \kappa) = r \restriction [\alpha, \kappa)$. So, the condition $p^* = r \restriction \alpha \cup q^* \restriction [\alpha, \kappa)$ is an $N$-generic direct extension of $p$, from $G$.} 
 By Claim \ref{claim:closure of j_tau''N}, $j_\tau(p^*)$ forces that 
 $M' = j_\tau\image N$ is closed under $j_\tau(F)$. 
 It is now clear that $M'$ satisfies condition (ii) in the ultrapower, 
 as $o^{j_\tau(\U)}(\kappa) = \tau$ and $M' \cap j_\tau(\kappa) = \kappa$. 
 Condition (i) holds as well, since $j(\po_\kappa) / \po_\kappa$ does not 
 introduce new ${<}\tau$-sequences to $j_\tau \image N$. 
 Therefore, it remains to verify that $j_\tau(p^*)$
 forces $M'$ to satisfy condition (iii). 
 For every $A \in M' \cap j_\tau(\F_\kappa)$ there is some $B \in \F_\kappa$ such that $A = j_\tau(B)$. In particular, $A \cap \kappa = B \in \F_\kappa$ and $\kappa \in A$. 
 Since $\F_\kappa \subseteq \bigcap_{i \leq \tau} U_{\kappa,i}$ (which is $\mathcal{F}_\kappa^{M_\tau}$), it follows form remark \ref{rmk1} that for every generic filter $G^* \subseteq j_\tau(\po)$
 over $M_\tau$, if $b_\kappa$ is the $G^*$-induced $\qo_\kappa^\tau$ cofinal generic sequence, then it is almost contained in $B = A \cap \kappa$.
\end{proof}

\begin{proposition}\label{prop1}
 $\co_{\F_\kappa}$ is $\kappa$-distributive.
\end{proposition}
\begin{proof}
 Since $\kappa = \lambda^+$ in $V[G\ast H]$, we need to check that the intersection of every set $\{ D_i \mid i < \lambda\}$ of $\lambda$-many dense open subsets of $\co_{\F_\kappa}$ is dense.  Pick some regular cardinal $\theta  > \kappa^+$ such that $\po, \co_{\F_\kappa}, \{D_i \mid i < \lambda\} \in H_\theta$. 
 By Lemma \ref{Lem2}, for every condition $x \in \co_{\F}$ there exists an elementary substructure $M \elem H_\theta$ of cardinality $< \kappa$, with $x, \po,\co_{\F_\kappa},\{D_i \mid i < \lambda\} \in M$ and which further satisfies (i) $M^{<\lambda} \subseteq M$; (ii) $\sup(M \cap \kappa) = \alpha$ has $o^{\U}(\alpha) = \lambda$; and 
 (iii) $\alpha \in A$ and $b_\alpha$ is almost contained in $A$ for every  $A \in \F_\kappa \cap M$.
 
 Let $\l \alpha_i \mid i < \lambda\r$ be an increasing enumeration of $b_\alpha$. 
 We construct by induction an increasing sequence of extensions $\l x_j \mid j < \lambda\r$ of $x$, together with an increasing sub-sequence $\l \alpha_{i_j} \mid j < \lambda\r$ of $b_\alpha$ such that $x_{j+1} \in D_j$ for every $j < \lambda$, and  $\{\alpha\} \cup \{ \alpha_{i_j} \mid j > j^*\} \subseteq A^{x_{j^*}}$ for all $j^* < \lambda$. 
 For notational simplicity, denote $x$ by $x_{-1}$.  Given a condition $x_j \in M$ with a suitable $\alpha_j$ as above, 
 we take $x_{j+1}\in D_{j+1}$ to be an extension of $x_j$ with $\max(c^{x_{j+1}}) > \alpha_{i_j}$.  Since $A^{x_{j+1}} \in M \cap \F_{\kappa}$ we can use (iii) and get that $\alpha \in A^{x_{j+1}}$ and
 there exist some $i' > i_j$ such that $\{ \alpha_i \mid i >i'\} \subseteq A^{x_{j+1}}$. Take $i_{j+1} < \lambda$ to be the minimal such $i' > i_j$.
 It remains to show that the construction goes through at limit stages $\delta \leq \lambda$. 
 Given $\l x_j \mid j < \delta\r$ we define $i_\delta = \sup_{j < \delta} \alpha_{i_j}$. It is clear from our construction at successor steps that
 $\alpha_{i_\delta} = \sup_{j < \delta} \max(c^{x_j})$ and $\alpha_{i_\delta} \in A^{x_j}$ for every $j < \delta$. 
 It follows that the condition $x_\delta = \l \{ \alpha_\delta \} \cup \bigcup_{j < \delta} c^{x_j}, \bigcap_{j<\delta} A^{x_j}\r$ satisfies the desirable conditions. Moreover if $\delta < \lambda$ then $x_\delta \in M$ since $M$ is closed under $<\lambda$-sequences. 
 
 Since the limit construction goes through at stage $\lambda$ as well (although not producing a condition in $M$), the limit condition $x_\lambda$ is an extension of $x$, and belongs to $\bigcap_{j <\lambda} D_j$. 
\end{proof}

The argument of the proof of lemma \ref{Lem0:CHomog} for $\co_U$ applies to $\co_{\F_\kappa}$ as well.
\begin{lemma}\label{Lem:CHomog}
 $\co_{\F_\kappa}$ is cone homogeneous.
\end{lemma}

\begin{theorem}\label{Thm-1stmeas}
In the generic extension by $\po * \col(\lambda,<\kappa) * \co_{\F_\kappa}$, $\kappa$ is strongly measurable.
\end{theorem}
\begin{proof}
 Suppose $G(\co_{\F_\kappa}) \subseteq\co_{\F_\kappa}$ be a generic filter over $V[G\ast H]$.
We may identify $G(\co_{\F_\kappa})$ with its derived generic closed and unbounded set \[C = \bigcup \{ c \mid \exists A \l c,A\r \in G(\co_{\F_\kappa})\}.\]
By a standard density argument we have that for every set $X \subseteq \kappa$ in $V$, 
if $X \notin U_{\kappa,\tau}$ for all $\tau \leq \lambda$ then $|C \cap X| < \kappa$.

We conclude that for $X \subseteq\kappa$ in $V$ to be stationary in $V[G \ast H \ast C]$ it must belong to $U_{\kappa,\tau}$ for some $\tau \leq \lambda$.
It follows that if $\l S_i \mid i <\eta\r \subseteq V$ is a partition of $\kappa$ into disjoint sets which are stationary in $V[G \ast H \ast C]$ then $|\eta| \leq \lambda$. 
Moreover, since $\kappa$ is inaccessible in $V$ we have $(2^\eta)^V < \kappa$.  

Finally, we know that each poset $\po$, $\col(\lambda,<\kappa)$, and $\co_{\F_\kappa}$ is forced in turn to be cone homogeneous and clearly definable using parameters from the ground model. Therefore
$\po* \col(\lambda,<\kappa) \ast \co_{\F_\kappa}$ is cone homogeneous, and therefore $\HOD^{V[G \ast H \ast C]} \subseteq V$. The claim follows. 
\end{proof}
The result in this section is weaker than the result of section \ref{Sec-omega1}, since the club filter is not an ultrafilter in $\HOD$. Since the club filter restricted to $S^{\omega_2}_{\omega}$ is an ultrafilter in a model of $\AD + V = L(\mathbb{R})$, one can force with the $\mathbb{P}_{max}$ forcing and obtain a generic extension in which the club filter restricted to $S^{\omega_2}_{\omega}$ is an ultrafilter in $\HOD$.\footnote{We would like to thank the referee for pointing us to this fact.}  

\begin{question}
Is it consistent that the club filter restricted to $S^{\lambda}_\omega$ is an ultrafilter in $\HOD$ for a regular cardinal $\lambda > \aleph_2$?
\end{question}

By the general behavior of covering arguments, it is possible that the consistency strength of $\omega_2$ being $\omega$-strongly measurable in $\HOD$ might be much lower than the same property for other successor of a regular cardinal and even be as low as a single measurable cardinal. 
\begin{question}
What is the consistency strength of $\omega_2$ being $\omega$-strongly measurable in $\HOD$? 
\end{question}

\newpage
\section{Many $\omega$-strongly measurable cardinals}\label{Sec-Many}
Suppose that $\U$ is a coherent sequence of normal measures so that $\lambda < \kappa$ are regular cardinals and $o^{\U}(\kappa) = \lambda + 1$ and that the first measure in $\U$ is on a cardinal strictly greater than $\lambda$. 
Let $\po^{\U}$ be the non-stationary support iteration of Prikry/Magidor forcing from \cite{bunger}, and $\co_{\F^{\U}_\kappa}$ be the $\po^{\U}* \col(\lambda,<\kappa)$-name of the associated 
diagonalizing club forcing for the filter $\F^{\U}_\kappa = \bigcap_{\tau \leq \lambda} U_{\kappa,\lambda}$ on $\kappa$. 
In the next section we construct a $\po^{\U} * \col(\lambda,<\kappa)$-name of a Prikry-type forcing notion $\bar{\co}_{\F^{\U}_\kappa}$, which is equivalent to $\co_{\F_\kappa}$, and its direct extension order is $\lambda$-closed. We will use that as a black box in this section.

\begin{definition}
  Denote the post $\po^{\U} * \col(\lambda,<\kappa) * \bar{\co}_{\F_\kappa}$ by $\qo[\U]$. 
\end{definition}

We have shown in the previous section that $\qo[\U]$ is cone homogeneous and equivalent as a forcing notion to 
the iteration $\po^{\U} \ast \col(\lambda,<\kappa) \ast \co_{\F_\kappa}$. 
By theorem \ref{Thm-1stmeas} we conclude that $\kappa$ is strongly measurable in the generic extension by $\qo[\U]$. \vskip\bigskipamount

In what follows, we would like to view $\qo[\U]$ as a Prikry-type forcing whose direct extension order is $\lambda$-closed. 
This is easily possible since $\qo[\U]$ is an iteration of three posets, each of which can be seen as a Prikry-forcing whose direct extension order is $\lambda$-closed (for $\col(\lambda,<\kappa)$ we identify the direct extension order with the standard order of the poset). \vskip\bigskipamount
We finally turn to prove our main result. 
%\begin{theorem}
%Let $\theta$ be an inaccessible cardinal such that there is a coherent sequence $\mathcal{U}$ for which $\sup \{o^{\mathcal{U}}(\lambda) \mid \lambda < \theta\} = \theta$. 

%Then, there is a generic extension in which $V[G]_\theta$ satisfies that for each regular cardinal $\kappa$, $\kappa^{+}$ is inaccessible and $(\kappa^{+}, \kappa)$-strongly measurable in $\HOD$. Moreover, it is $(S^{\kappa^{+}}_{\kappa}, 1)$-strongly measurable.
%\end{theorem}
%Since we start with a large cardinal axiom much below a Woodin cardinal, the weak covering lemma holds and $\HOD$ computes correctly successors of singular cardinals. 
\begin{proof}(Theorem \ref{THM1})\vskip\bigskipamount
To simplify our arguments, we work over a minimal Mitchell model $V = L[\U]$ with a coherent sequence of measures $\U$ witnessing the assumed large cardinal assumption. Therefore 
$\theta$ is the least inaccessible cardinal in $V$ for which $\{ o(\kappa) \mid \kappa < \theta\}$ is unbounded in $\theta$. 
We note that all normal measures in this model appear on the main sequence $\U$, in particular, $o(\kappa) = o^{\U}(\kappa)$ for all $\kappa$. 
We also record here that by Mitchell Covering Theorem and the fact $\theta$ is not measurable,  there is no generic extension of $V = L[\U]$ which preserves the cardinals below $\theta$ and changes the cofinality of $\theta$.
Similarly, the Mitchell Covering Theorem guarantees that generic extensions of $V = L[\U]$ satisfy the Weak Covering Lemma with respect to $V$, which implies that successors of singular cardinals cannot be collapsed. 

%Let us assume, for simplicity that $\theta$ is the least inaccessible cardinal that satisfies the conditions of the lemma. 
%We further assume that the Weak Covering Lemma holds and in particular, if $\kappa$ is a successor cardinal in the ground model and $\cf \kappa = \mu$ in some generic extension then $|\kappa| = \mu$.

Let $\langle \kappa_\alpha \mid \alpha < \theta\rangle$ be an increasing sequence of cardinals below $\theta$, which satisfies the following conditions:
\begin{enumerate}
\item $\kappa_0 = \omega$, $\kappa_1$ is the least measurable,
\item for a limit ordinal $\alpha$, $\kappa_\alpha = \left(\sup_{\beta < \alpha} \kappa_{\beta}\right)^+$,
\item for a successor ordinal $\alpha + 1$ let $\kappa_{\alpha + 1}$ is the least cardinal such that $o^{\mathcal{U^\alpha}}(\kappa_{\alpha + 1}) = \kappa_{\alpha} + 1$, for the coherent sequence of measures $\U^\alpha = \U\uhr_{(\kappa_\alpha,\kappa_{\alpha+1}]}$. In particular, the first measure of the sequence $\U^\alpha$ has critical point $> \kappa_\alpha$.
\end{enumerate}

We define by induction on $\alpha < \theta$ a Magidor iteration $\po = \l \po_\alpha, \qo_\alpha \mid \alpha < \theta\r$ of Prikry type forcings. Our description of the Magidor style iteration follows  Gitik's handbook chapter, \cite{Gitik-HB}. We recall that conditions are sequences of the form $\langle q_\alpha \mid \alpha < \theta\rangle$ where only finitely many coordinates are not a direct extension of the weakest condition $0_{\qo_\alpha}$. 
Let $\mathbb{Q}_{0}$ be $\col(\omega, <\kappa_1) * \co^*_{\F_{\kappa_1}}$, where $\F_{\kappa_1}$ is the filter generated from the normal measure on $\kappa_1$. For $\alpha > 0$, we define $\qo_{\alpha} = \qo[\U^\alpha]$. 

The coherent sequence $\U^\alpha$ from $L[\U]$ uniquely extends in a generic extension by $\po_\alpha$, and can therefore be used to force with $\qo[\U^\alpha]$. This is because as $L[\U]$ satisfies the $\GCH$, we have that $|\po_\alpha| \leq \kappa_\alpha$ and all measures of $\U^\alpha$ are assumed to have critical points strictly above $\kappa_\alpha$. 
It is clear from our definitions that $\qo_\alpha$ satisfies the  Prikry Property, that  its direct extension order  is $\kappa_\alpha$-closed, and that $\mathbb{Q}_\alpha$ is forced to be cone homogeneous.
%and has the property that any subset of $\kappa_\alpha$ which is introduced by $\mathbb{Q}_{\alpha}$ is already introduced by the first two steps $\mathbb{P}^{\mathcal{U}\restriction [\kappa_\alpha + 1, \kappa_{\alpha + 1})} * \col(\kappa_\alpha, <\kappa_{\alpha + 1})$. 

By the general theory of Magidor iteration of Prikry type posets, the iteration $\mathbb{P}_{\theta}/ \po_1$ also satisfy the Prikry Property. Moreover, for every $\alpha < \theta$, $\mathbb{P}_{\theta} / \mathbb{P}_\alpha$ has the Prikry Property in the generic extension by $\mathbb{P}_{\alpha}$, and its direct extension order is $\kappa_\alpha$-closed (see \cite{Gitik-HB} for details). 

\begin{claim}
Every bounded subset of $\kappa_{\alpha}$ is introduced by $\po_{\alpha}$. Moreover, in the generic extension by $\mathbb{P}_{\theta}$, $\kappa_{\alpha}$ is a regular cardinal for all $\alpha < \theta$.
\end{claim}
\begin{proof}
The first assertion is an immediate consequence of the fact $\po_\theta/ \po_\alpha$ satisfies the Prikry Property and its direct extension order is $\kappa_\alpha$ closed. 
It follows that in order to show that all cardinal $\kappa_\alpha$ remain regular in a generic extension by $\po_\theta$, it suffices to show that 
$\kappa_\alpha$ remains regular in the intermediate generic extension by $\po_\alpha$. We prove the last assertion by induction on $\alpha < \theta$. 

For a limit ordinal $\alpha$, the assertion follows from the fact that the generic extension by $\po_\alpha$ satisfies the Weak Covering property with respect to the ground model $V = L[\U]$. Indeed, $\kappa_{\alpha} = \left(\sup_{\beta < \alpha} \kappa_{\beta}\right)^+$ cannot be collapsed without collapsing a tail of the cardinals $\kappa_\beta$, $\beta < \alpha$, which would contradict our inductive assumption. 

Suppose now that $\alpha$ is a successor ordinal. Then the forcing $\po_{\alpha}$ naturally breaks into two parts $\po_\alpha \cong \po_{\alpha - 1} * \qo_{\alpha-1}$. The size of $\po_{\alpha - 1}$ is $\left(2^{\kappa_{\alpha-1}}\right)^V < \kappa_\alpha$, and cannot singularize $\kappa_{\alpha}$. The second poset $\qo_{\alpha-1}$ does not collapse $\kappa_{\alpha}$ by Proposition \ref{prop1}. Note that in order to apply the result of Proposition \ref{prop1} we use our inductive hypothesis that $\kappa_{\alpha-1}$ remain regular in a generic extension by $\po_{\alpha-1}$.
\end{proof}

\begin{claim}
In the generic extension $\theta$ is regular.
\end{claim}
\begin{proof}
This follows from the Mitchell Covering Theorem and the smallness assumption of $\theta$, as was mentioned at the beginning of the proof. 
\end{proof}

\begin{claim}
 $\po$ is cone homogeneous.
\end{claim}
\begin{proof}
It suffices to verify conditions (i),(ii) of Lemma \ref{FACT:homogiter} hold for every $\alpha < \kappa$. 
(i) holds for $\qo_\alpha = \qo[\U\uhr_{(\kappa_\alpha,\kappa_{\alpha+1}]}]$ since $\qo_\alpha,\leq_{\qo_\alpha},\leq^*_{\qo_\alpha}$ are clearly definable in $V = L[\U]$ from $\U$, $\kappa_\alpha$, and $\kappa_{\alpha+1}$.
The fact $(\qo_\alpha,\leq_{\qo_\alpha},\leq^*_{\qo_\alpha})$ is an immediate consequence of  Lemma \ref{Lem:Q[U].weakly.homogeneous}.
\end{proof}

Let $G_\theta \subseteq \po_\theta$ be a generic filter over  $V$. We conclude that $\HOD^{V[G_\theta]} \subseteq V$.\footnote{As a matter of fact $\HOD^{V[G_\theta]} = V$ since $V = L[\U]$ is ordinal definable in $V[G_\theta]$. We will not used this fact here.}  Moreover, for each $\alpha < \kappa$, the $\qo[\U^\alpha]$ generic filter induced by $G_\theta$ guarantees that $\kappa_{\alpha+1} = (\kappa_\alpha^+)^{V[G_\theta]}$ and that $\kappa_\alpha$ is strongly measurable in $\HOD^{V[G_\theta]}$. 
It follows that all successors of regular cardinals below $\theta$ in $V[G_\theta]$ are strongly measurable in $\HOD$. Since $\theta$ remains strongly inaccessible in $V[G_\theta]$ and all the relevant witnessing objects clearly belong to $V_\theta^{V[G_\theta]}$, we conclude that 
in $V_\theta^{V[G_\theta]}$, all successors of regular cardinals are strongly measurable. 
\end{proof}
\section{Embedding $\co_{\F_\kappa}$ in suitable Prikry-type forcings}\label{Sec-PrikryEmbedding}
The method of the previous section can be iterated finitely many times in order to get finitely many successive $\omega$-strongly measurable cardinals. In order to get a global result (or even just infinitely many $\omega$-strongly measurables) we need to have a preservation of distributivity under iterations. 

This is, in general, a difficult task. One way to obtain this is by shifting our goal from preserving distributivity into preserving the Prikry Property. There are several ways to iterate Prikry type forcings and preserve the Prikry Property as well as the closure properties of the direct extension. Thus, embedding the distributive forcings into a Prikry type forcing can be used in order to get a suitable distributivity of the iteration. Usually, in order to achieve this, some strong compactness assumption is made that enables one to embed any sufficiently distributive forcing into a Prikry type forcing. See \cite{Gitik-Compact, BenHamouHayutGitik}, for some examples for the consistency strength of such constructions.   

Our goal is to embed $\co_{\F_\kappa}$ into a Prikry type forcing without increasing our large cardinal hypothesis from $o(\kappa)= \lambda+1$. For this, our approach follows the finer technique, introduced by Gitik in \cite{Gitik-ClubOnRegs}.

This section is devoted to prove the following technical lemma:
\begin{proposition}\label{proposition:embedding-C-into-prikry-type}
Let us assume that $\mathcal{U}$ is a coherent measure sequence witnessing $o(\kappa) = \lambda+1$. Let $\mathbb{P}^\mathcal{U}$ be the non-stationary support iteration of Subsection \ref{ssec-non-stationary-iteration}. Then in $\mathbb{P}^{\mathcal{U}} \ast \col(\lambda,<\kappa)$ there is a Prikry-type forcing notion $\bar{\co}_{\F_\kappa}$, whose direct extension order is $\lambda$-closed and it has a dense subset isomorphic to $\mathbb{C}_{\F_\kappa}$. 

Moreover, both orders $\leq$ and $\leq^*$ witness the forcing $\mathbb{P}^{\mathcal{U}} \ast \col(\lambda,<\kappa) \ast \bar{\co}_{\F_\kappa}$ to be cone homogeneous.
\end{proposition}
The proof of the first part of the proposition is given in Corollary \ref{cor:bar-C-prikry-type} and the proof of the moreover part appears in Lemma \ref{Lem:Q[U].weakly.homogeneous}.
Let us sketch the main ideas behind to proof of the proposition. In order to construct a Prikry type forcing that projects onto $\mathbb{C}_{\F_\kappa}$ we first work in the generic extension in which $\kappa$ is singularized to be of cofinality $\lambda$. In this model, the Magidor sequence is already a closed unbounded set that diagonalizes the filter $\F_\kappa$, so we can use it as a guide to the generic of $\co_{\F_\kappa}$. This means that there is a projection from the generic extension by the singularizing forcing iterated by a $\lambda$-closed forcing onto $\co_{\F_\kappa}$. $\bar{\co}_{\F_\kappa}$ is obtained by "forgetting" the Magidor sequence and keeping the diagonalizing club. A technical issue that arise when trying to pull up this strategy is that the singularizing forcing must be defined \emph{after} the cardinals between $\lambda$ and $\kappa$ were collapsed, and a major part of this section is devoted to developing this forcing. 

The rest of this section is organized as follows: in subsection \ref{subsection:q-star} we review the basic construction and properties of the tree Prikry-type forcings $\qo^\tau_\kappa$, $\tau \leq o^{\U}(\kappa)$ which is defined in the generic extension by $\po$. Then, we introduce a filter-based variant $\qo_{\kappa,\tau}^*$ to be forced over a generic extension $V[G \ast H]$ by  $\po \ast \col(\lambda,<\kappa)$ extension  $V[G \ast H]$ of $V$, where $\kappa = \lambda^+$ is no longer measurable. 

In subsection \ref{subsection:c-star}, we use the posets $\qo_{\kappa,\tau}^*$, $\tau \leq \lambda$ in order to introduce a forcing equivalent 
$\bar{\co}_{\F_\kappa}$ of $\co_{\F_\kappa}$ with a dense Prikry-type sub-forcing $\co^*_{\F_\kappa}$ whose direct extension order is $\lambda$-closed. 

This completes the proof of Section \ref{Sec-Many}, as the posets $\bar{\co}_{\F_\kappa}$ can be iterated on different cardinals to construct models with many $\omega$-strongly measurable cardinals. 

\subsection{The forcing $\qo^*_{\kappa,\tau}$}\label{subsection:q-star}
We turn back to consider our forcing scenario with $\F_{\kappa}$ over $V[G\ast H]$, where $G \subseteq \po$ is generic over $V$, and $H \subseteq \col(\lambda,<\kappa)$ is generic over
$V[G]$. Recall that $o^{\U}(\kappa) = \lambda+1$ and that each measure $U_{\kappa,\tau}$, $\tau \leq \lambda$ in $V$, extends in $V[G]$ to $U_{\kappa,\tau}(t)$, where $t$ is $\tau$-coherent. 
For each such $\tau \leq \lambda$ and $t$, let $j_{\kappa,\tau,t}\colon V[G] \to M_{\kappa,\tau,t}$ be the ultrapower embedding of $V[G]$ by $U_{\kappa,\tau}(t)$. 

Moving to the further generic extension $V[G \ast H]$ of $V[G]$,  $\kappa$ is no longer measurable. Let $F_{\kappa,\tau}(t)$ denote the filter generated by 
$U_{\kappa,\tau}(t)$ on $\power(\kappa)^{V[G]}$ and $F_{\kappa,\tau}(t)^+$ denote the poset on $\power(\kappa)^{V[G]}$ of  $F_{\kappa,\tau}(t)$ positive sets where a set $A$ is stronger than 
$B$ if $A \setminus B$ belongs to the dual ideal of $F_{\kappa,\tau}(t)$. 

By further forcing with the collapse quotient \[\mathbb{R}_{\tau} = \col(\lambda,<j_{\kappa,\tau,t}(\kappa))/H \cong \col(\lambda, [\kappa, j_{\kappa,\tau,t}(\kappa)),\] over $V[G \ast H]$, producing a generic filter $H_{\tau}^* \subseteq \col(\lambda,<j_{\kappa,\tau,t}(\kappa))$, with $H_{\tau}^* \uhr \col(\lambda,<\kappa) = H$,  the elementary embedding $j_{\kappa,\tau,t}$ extends into 
\[j^*_{\kappa,\tau,t} \colon V[G \ast H] \to M_{\kappa,\tau,t}[H_\tau^*].\] 

In turn, the embedding $j^*_{\kappa,\tau,t}$ generates a $V[G \ast H]$ ultrafilter $U_{\kappa,\tau}(t)^* \subseteq F_{\kappa,\tau}(t)^+$, which is an $F_{\kappa,\tau}(t)^+$-generic ultrafilter over $V[G \ast H]$, by standard arguments connecting forcing with positive sets and generic ultrapowers.\footnote{Indeed, one can verify that the trivial condition in $\col(\lambda, [\kappa, j_{\kappa,\tau,t}(\kappa))$ forces $\kappa \in j(\dot{X})$ if and only if there is a subset of $\dot{X}$ in $U_{\kappa,\tau}(t)$.} 
Since the poset $\mathbb{R}_{\tau} = \col(\lambda,<j_{\kappa,\tau,t}(\kappa))/H$ is $\lambda$-closed in $V[G \ast H]$, we have that $F_{\kappa,\tau}(t)^+$ has a $\lambda$-closed dense sub-forcing $D_{\kappa,\tau,t}$. Other examples of applications of Prikry-type forcings generated by ideals can be found in  \cite{MSI} and \cite{MSLB}.

It would be useful for our purposes to work with a concrete description of the sets in $D_{\kappa,\tau,t}$.
We proceed to introduce the relevant notions.
\begin{definition}
Let $G \subseteq \po$ be a generic filter over $V$. For each cardinal $\nu < \kappa$ with $o^{\U}(\nu) > 0$, let $b_\nu$ be the $G$-induced generic cofinal sequence in $\nu$.
 \begin{enumerate}
  \item Recall that every finite coherent sequence $t = \l \nu_0,\dots, \nu_{k-1}\r \in [\kappa]^{<\omega}$ in $V[G]$, has an assigned closed unbounded set $b_t = \cup_{i<k} (b_{\nu_i} \cup \{\nu_i\})$. 
  For a coherent sequence $t$ and a finite set of ordinals $s \in [\min(t)]^{<\omega}$, we define 
  $\pi^s(t) = \min(b_t \setminus (\max(s)+1))$. 
  When $t = \l \nu\r$ has a single element, we will often abuse this definition and write $\pi^s(\nu)$ for $\pi^s(\l \nu \r)$. \vskip\bigskipamount
 % \[ \pi^s_0(\nu) = 
 %  \begin{cases}
 %   \min(b_t \setminus \max(s) +1) \mbox{ if } o^{\U}(\nu) > 0\\
 %   \nu \mbox{otherwise. }
 %  \end{cases}
 % \]
 For every $\eta$, the function
 \[\pi^s_{\eta}(\nu) = \min( \{ \mu \in b_\nu \setminus (\max(s)+1) \mid o^{\U}(\mu) = \eta\})\] 
 defines a Rudin-Keisler projection from 
 $U_{\kappa,\tau}(s)$ to $U_{\kappa,\eta}(s)$, for all $ \tau > \eta$. In particular $\pi^s = \pi^s_0 \colon \kappa \to \kappa$ is a Rudin-Keisler projection of  $U_{\kappa,\tau}(s)$ to its normal projected measure 
 $U_{\kappa,0}(s)$, for every $\tau \geq 0$. \vskip\bigskipamount
 
 % \item Let $s \in [\kappa]^{<\omega}$, $q \in \col(\lambda,<\kappa)$,  $A \subseteq \kappa \setminus (\max(s)+1)$, 
 % and $Q\colon A \to \col(\lambda,<\kappa)$. We say that $Q$ is \textbf{$\la s,q \r$-suitable} if
 % $q =  Q(\nu)\uhr \lambda \times \pi^s_0(\nu)  \in \col(\lambda,<\pi^s_0(\nu))$ 
 %  for all $\nu \in A$. 

  \item Let $T \subseteq [\kappa]^{<\omega}$ a tree, $t \in [\kappa]^{<\omega}$, and  $Q\colon T \to \col(\lambda,<\kappa)$ be a function.
  We say that $Q$ is \textbf{$(T,t)$-suitable} if for every 
  $s \in T$ we have
  \begin{itemize}
   \item $Q(s) \in \col(\lambda,<\kappa)$, and  
   \item for every $s' \in T$ that extends $s$,  $Q(s')\uhr \lambda \times \pi_0^{t \fr s}(s') = Q(s)$.
  \end{itemize}
  
  For every $s \in T$ we define $Q_s$ to be the induced function on  $T_s = \{ r \in [\kappa]^{<\omega} \mid s \fr r \in T\}$, given by 
  \[Q_s(r) = Q(s \fr r).\]
  
  \item Suppose that $H \subseteq \col(\lambda,<\kappa)$ is generic over $V[G]$,
  $T,Q \in V[G]$ as above, and let $A = \succ_T(\emptyset)$. We define in $V[G][H]$ the set $Q$-generic restriction of $A$ with respect to $H$, to be the set 
  \[A^H_Q = \{ \nu \in A \mid Q(\nu) \in H\}.\] 
 \end{enumerate}
\end{definition}

Let $G \subseteq \po$ be generic over $V$ and $H \subseteq \col(\lambda,<\kappa)$ be a generic over $V[G]$. 
In $V[G]$, for all $\tau \leq \lambda$ and $t \in [\kappa]^{<\omega}$, 
the function $\pi^t = \pi^t_0$ represents $\kappa$ in the ultrapower by $U_{\kappa,\tau}(t)$. 
It is therefore immediate from our definition of $D_{\kappa,\tau,t} \subseteq F_{\kappa,\tau}(t)^+$ that sets in $D_{\kappa,\tau,t}$ are of the form 
as $A^H_{r} = \{ \nu \in A \mid r(\nu) \in H\}$ where 
 $A \in U_{\kappa,\tau}(t)$, and $r\colon A \to \col(\lambda,<\kappa)$ 
satisfying $r(\nu) \in \col(\lambda,[\pi^t(\nu),\kappa))$ for all $\nu \in A$.\footnote{i.e., $\dom(r(\nu)) \subseteq \lambda \times (\kappa \setminus \pi^t(\nu))$.}

%It is an immediate consequence of the 
%last fact and the definition of the families $D_{\kappa,\tau,t}$ 
%that in $V[G][H]$,
%$D_{\kappa,\tau,t}$ consists of set of the form $A^H_Q$ for $A \in U_{\kappa,\tau}(t)$ and 
%$Q \colon A \to \col(\lambda,<\kappa)$ which is $t$-suitable.

%Moreover, since $\langle j_{\kappa,\tau}(\kappa) \mid \tau \leq \lambda\rangle$ is an increasing sequence of ordinals, each forcing notions $\mathbb{R}_{\tau}$  embeds into $\mathbb{R}_{\tau'}$ for $\tau < \tau'$. Thus, in the generic extension by $H^*_{\lambda}$ we have a Mitchell increasing sequence of $V[G\ast H]$-measures, $U^*_{\kappa,\tau, t}$. 
  
We use these facts to introduce a variant of Gitik's forcing $\mathbb{Q}_{\kappa,\tau}$ in $V[G \ast H]$. The following poset $\qo^*_{\kappa,\tau}$ collapses cardinals up to $\kappa^+$ to $\lambda$ and adds a cofinal Magidor sequence $b_{\kappa}$ of length $\omega^{\tau}$ to $\kappa$, which diagonalizes the filter $\bigcap_{\tau' < \tau} U_{\kappa,\tau'}$ (i.e., $b_\kappa$ is almost contained in each filter set). 
%and thus allow us to apply similar arguments to the ones in Section \ref{Sec-Cforcing}.

\begin{definition}\label{def:Q*poset}
In $V[G \ast H]$ the forcing $\qo^*_{\kappa,\tau}$ consists of all $(t, T, Q) \in V[G]$ such that:
\begin{enumerate}
\item $t$ is a $\tau$-coherent finite sequence of ordinals below $\kappa$,
\item $T$ is a tree of $\tau$-coherent finite sequences with stem $t$,
\item $Q$ is a $(T,t)$-suitable function, 
\item $Q(\emptyset) \in H$, and
\item\label{condition:agreement}  For every $s, s' \in T$, if $b_{t \fr s} = b_{t \fr s'}$ then $Q(s) = Q(s')$. 
%for every $s' \in T$ and $\eta \leq \tau$. If $s' = s_0 \fr s_1$, where $s_0$ is the maximal initial segment of $s'$ which contains an ordinal $\nu'$ with $o^{\U}(\nu') \geq \eta$ then
%\[ [\nu \mapsto Q_{s_0}( s_1 \fr \l\nu\r)]_{U_{\kappa,\eta}(t \fr s')} = [\nu\mapsto Q_{s_0}(\nu)]_{U_{\kappa,\eta}(t\fr s_0)}. \] 
\end{enumerate}
\end{definition}
As in $\qo_{\kappa,\tau}$, we identify two condition $(t,T,Q),(t',T,Q) \in \qo_{\kappa,\tau}^*$, whenever $b_t= b_{t'}$. 

% \begin{remark}\label{remark:stronger agreement condition}
% Condition (\ref{condition:agreement}) in Definition \ref{def:Q*poset} can be strengthen on a $\leq^*$-dense set to the condition that for $s, s'\in T$, if $b_{s\fr t} = b_{s' \fr t}$ then $Q(s) = Q(s')$. Indeed, by inductively narrowing down the tree $T$, we may restrict for every $s_0 \fr s_1$ the set of successors of $s_0 \fr s_1$ of Mitchell order $\eta$ to all $\nu$ such that $Q_{s_0}(s_1 \fr \l\nu \r) = Q_{s_0}(\nu)$. By Condition (\ref{condition:agreement}), the remaining tree is large.
% \end{remark}

The \textbf{direct extension} ordering of $\qo^*_{\kappa,\tau}$ naturally extends the direct extension ordering of $\qo_{\kappa,\tau}$. Namely, for two conditions $(t,T,Q)$, $(t^*,T^*,Q^*)$ of $\qo^*_{\kappa,\tau}$, we have $(t,T,Q) \leq^* (t^*,T^*,Q^*)$ if $(t,T)\leq^*_{\qo_{\kappa,\tau}} (t^*,T^*)$ and $Q^*(s) \geq Q(s)$ for every $s \in T^*$. \vskip\bigskipamount

We observe that the direct extension order $\leq^*$ is $\lambda$-closed in $V[G\ast H]$. For this, note that it is immediate from the definition above that 
the  partial order $\tilde\leq\in V[G]$, obtained from $\leq^*$ by removing the requirement $Q(\emptyset)\in H$, belongs to $V[G]$ and is clearly $\lambda$-closed in both $V[G]$ and $V[G\ast H]$ (note that the two generic extensions agree on sequences of length $<\lambda$).
Then, as $\leq^*$ is equivalent to the restriction of $\tilde{\leq}$ to a $\lambda$-closed set (which is essentially $H$), it remains $\lambda$-closed in $V[G\ast H]$.\vskip\bigskipamount

The \textbf{end-extension} ordering of $\qo^*_{\kappa,\tau}$ is based the restriction of the end-extension of $\qo_{\kappa,\tau}$ to the $Q$-generic restriction of $T$ with respect to $H$. Namely, for a condition $p = (t, T, Q)$, \textbf{the only values $\nu \in \suc_T(\emptyset)$ which are allowed to used when taking a one-point extension, are $\nu \in \suc_T(\emptyset)^H_Q$.}\footnote{I.e., $\nu \in A^H_Q$ for $A=\suc_T(\emptyset)$.} In this case, the resulting one-point extension is defined to be $p \fr \l \nu\r = (t \cup \{\nu\}, T_{\l \nu \r}, Q_{\l \nu \r})$. In general, for a sequence $r = \l \nu_0,\dots, \nu_{k-1} \r \in T$, the end extension of $p$ by $r$, denoted $p \fr r$, is the one obtained by taking a sequence of one-point extensions by $\nu_0,\dots,\nu_{k-1}$, in turn. %\newline
\vskip\bigskipamount

 We note that although $\qo_{\kappa,\tau}^*$ depends on the collapse generic $H$, and is fully defined only in $V[G\ast H]$, we still have that $\qo_{\kappa,\tau}^* \subseteq V[G]$. Moreover, dropping the requirement $Q(\emptyset) \in H$ in the definition of conditions $p =(t,T,Q) \in \qo_{\kappa,\tau}^*$ allows us to examine conditions $(t,T,Q)$ and evaluate possible direct extensions and one-point extensions in $V[G]$.
For example, working in $V[G]$, we can consider possible one-point extensions $p \fr \l \nu \r$ of a condition $p= (t,T,Q)$ by an arbitrary $\nu \in \suc_T(\emptyset)$. Although eventually, in $V[G\ast H]$, $p$ will be a valid condition only if $Q(\emptyset) \in H$, and  $p \fr \l \nu \r$ will form valid extensions of $p$ only for a $D_{\kappa,\tau,t}$-positive set of ordinals $\nu \in \suc_T(\emptyset)$, it is still possible to decide certain properties of such extensions on a measure one set of $U_{\kappa,\tau}(t)$ in $V[G]$. %Furthremore, the additional generic collapse information of $H$ can be handled using suitable functions $Q$, in  $V[G]$. 
This approach of arguing from $V[G]$ about the poset $\qo_{\kappa,\tau}^*$ in $V[G\ast H]$ plays a significant role in our proof below, showing that $\qo_{\kappa,\tau}^*$ satisfies the Prikry Property.

\begin{remark}
The forcing $\col(\kappa < \lambda) * \qo^*_{\kappa, \tau}$ is isomorphic to the collection of all $(t, T, Q)$ that satisfy all requirements of Definition \ref{def:Q*poset} expect the fourth one, $Q(\emptyset) \in H$. Nevertheless, the decomposition into the collapse part and the singularization part would be more appropriate for our construction, as eventually $\qo^*_{\kappa,\tau}$ is used as merely an auxiliary forcing.
\end{remark}

The coherency requirements in definition \ref{def:Q*poset} allow us to obtain a natural amalgamation property, similar to the one satisfied by the $V[G]$ poset $\qo_{\kappa,\tau}$.

\begin{lemma}\label{Lem.amalgamation.in.Q*}
 Work in $V[G]$. Let $p = (t,T,Q)$, forced to be a condition in $\qo^*_{\kappa,\tau}$ by $q = Q(\emptyset)$, and $A = \suc_T(\emptyset)$. 
 For each $\eta < \tau$,  
 denote \[A(\eta) = A \cap \{ \nu \in A \mid o^{\U}(\nu) = \eta\} \in U_{\kappa,\eta}(t) .\]%\setminus (\bigcup_{\eta' \neq \eta} U_{\kappa,\eta'}(t)).\] 

 Suppose that there are $\bar{\eta} < \tau$, a set $A'(\bar{\eta}) \subseteq A(\bar{\eta})$ with $A'(\bar{\eta}) \in U_{\kappa,\bar{\eta}}(t)$, and a sequence  of conditions 
 $\l (t\cup\{\nu\} , T^\nu,Q^\nu) \mid \nu \in A'(\bar{\eta}) \r$,  such that 
 $(t\cup\{\nu\} , T^\nu,Q^\nu)$ is a direct extension of $p \fr \l \nu \r$ for all $\nu \in A'(\bar{\eta})$.  
 Then there exists a direct extension $p^* \geq^* p$, $p^* = (t,T^*,Q^*)$,
 such that the set $\{ (t\fr \l \nu \r,T^\nu,Q^\nu) \mid \nu \in A'(\bar{\eta}) \}$ is forced by $Q^*(\emptyset)$ to be predense above $p^*$. 
\end{lemma}

\begin{remark}\label{RMK:Qconstruction}

In the proof of the lemma we make use of several construction arguments involving trees $T$ associated to conditions $(t, T)$ in the poset $\qo_{\kappa,\tau}$
%($\eta \leq \tau$) 
from \cite{gitik-nonstionary-ideal}. We list these arguments and refer the reader to \cite{gitik-nonstionary-ideal} for proofs. \vskip\bigskipamount

For a finite sequence $s$, we write $o(s) = \max(\{o(\nu) \mid \nu \in s\} )$.\vskip\medskipamount
\begin{enumerate}
 \item Suppose that $(t,T)$ is a condition of $\qo_{\kappa,\tau}$ and $A'(\eta) \subseteq \succ_T(\emptyset)$ belongs to 
 $U_{\kappa,\eta}(t)$ for some $\eta < \tau$. Then there exists a sub-tree $T'$ of $T$, so that $(t,T') \in \qo_{\kappa,\tau}$ is a direct extension of
 $(t,T)$, and 
 \[ \{ \nu \in \succ_{T'}(\emptyset) \mid o(\nu) = \eta\} \subseteq A'(\eta) .\]
 
 Similarly, for every $s \in T$ and $A'_s(\eta)  \subseteq \succ_T(s)$ which 
 belongs to 
 $U_{\kappa,\eta}(t\fr s)$ there is a direct extension $(t,T')$ of $(t,T)$, 
 which only requires shrinking the tree $T$ above $s$ (i.e., shrinking  $T_s$) and in particular $s\in T'$, 
 so that $\{ \nu \in \succ_{T'}(s) \mid o(\nu) = \eta\} \subseteq A'_s(\eta)$.  \vskip\bigskipamount
 
 Furthermore, this construction can be naturally combined over different values $s \in T$. Namely, given a family $\{ A'_s(\eta) \mid s \in T \}$ of sets as above we can apply the same procedure, level by level, to the tree $T$, and obtain a sub-tree $T' \subseteq T$ with the property that 
 $(t,T') \in \qo_{\kappa,\tau}$ and for every $s \in T'$, 
 \[ \{ \nu \in \succ_{T'}(s) \mid o(\nu) = \eta\} \subseteq A'_s(\eta). \]
 %\vskip\bigskipamount
 \item For a condition $(t,T)$, $s \in T$, and $\eta < \tau$, there exists a direct extension $(t,T') \geq^* (t,T)$, which only requires shrinking the tree $T$ above $s$ such that 
 for all $s' \in T$ which end extends $s$, if there exists $\nu \in b_{s'}\setminus b_s$ such that $o(\nu) = \eta$,\footnote{this is equivalent to the existence of  $\mu \in s'\setminus s$ such that $o(\mu) \geq \eta$.} then $\nu' = \pi^{t \fr s}_{\eta}(s')$ (the minimal such $\nu$) belongs to $\succ_{T'}(s)$. 
 Repeating this construction, level by level, produces a direct extension $(t,T')$ of $(t,T)$ satisfying that for every $s \in T'$ and $s' \in T'$ which extends $s$,  $\nu' = \pi_\eta^{t \fr s}(s') \in \succ_{T'}(s) \cap \{ \nu < \kappa \mid o(\nu) = \eta\}$. \vskip\bigskipamount
  
 We note that if $s \in T'$ satisfies that $o(\mu) < \eta$ for all $\mu \in s$, then for every $\mu \in \succ_{T'}(s)$ with $o(\mu) \geq \eta$, we have  $\pi^{t\fr s}_\eta(\mu) = \pi^t_\eta(s \fr \l\mu\r)$, which by our assumption of $T'$ (applied to $s' = s \fr \l \mu\r$), implies that $\pi^t_\eta(s \fr \l\mu\r) \in \succ_{T'}(\emptyset)$.  It follows that \[\succ_{T'}(s) \cap \{\nu \mid o(\nu) = \eta\} \subseteq \succ_{T'}(\emptyset) \cap \{\nu \mid o(\nu) = \eta\}.\] 
 Since the former set belongs to $U_{\kappa,\eta}(t \fr s)$ we conclude that 
 \[\succ_{T'}(\emptyset) \cap \{\nu \mid o(\nu) = \eta\} \in U_{\kappa,\eta}(t \fr s)\]
  as well.
 
 The same consideration applies to any $s \in T'$ and $s'\in T'$ which extends $s$, and for which $o(\nu) < \eta$ for every $\nu \in s'\setminus s$, 
 and implies that \[\succ_{T'}(s) \cap \{\nu \mid o(\nu) = \eta\} \in U_{\kappa,\eta}(t \fr s').\] \vskip\bigskipamount
 
 \item Let $(t,T')$ be a condition as in the previous clause. 
 There exists a direct extension $(t,T^*)$ of $(t,T')$ such that for every $s' \in T^*$ for which $\nu' = \pi^{t}_{\eta}(s') \in b_{s'}\setminus b_t$ is defined, not only that \[\nu' \in \succ_{T^*}(\emptyset) \cap \{\nu \mid o(\nu) = \eta\},\] but further, there is  some $s'' \in T^*$ which extends $\l \nu'\r$ such that $b_{t \fr s''} = b_{t \fr  s'}$ and  $T^*_{s'} \subseteq T^*_{s''}$.  We note that it implies that the set \[\{ (t \cup \{ \nu \}, T^*_{\la \nu \ra}) \mid \nu \in \succ_{T^*}(\emptyset) , o(\nu) = \eta\}\] is predense above $(t,T^*)$.\vskip\bigskipamount
 
 \noindent
\end{enumerate}
 We turn to the proof of Lemma \ref{Lem.amalgamation.in.Q*}.
 
\end{remark}

\begin{proof}(Lemma \ref{Lem.amalgamation.in.Q*})\\
For $s \in T$ define $o(s) = \max(\{o(\nu) \mid \nu \in s\})$, 
and for a tree $T \subseteq [\kappa]^{<\omega}$, and $\eta$, $T(<\eta) = \{ s \in T \mid o(s) < \eta\}$.
Let $p = (t,T,Q)$, $A'(\bar{\eta})$, and $\l (t\cup\{\nu\} , T^\nu,Q^\nu) \mid \nu \in A'(\bar{\eta}) \r$, as in the statement of the Lemma.
By part (1) of Remark \ref{RMK:Qconstruction} above, we may assume (by reducing to a suitable sub-tree) that $\succ_T(\emptyset) \cap \{ \nu \mid o(\nu) = \bar{\eta}\} 
\subseteq A'(\bar{\eta})$. 
Furthermore, by part (2) of the remark, we may further assume that $A'(\bar{\eta}) \in U_{\kappa,\bar{\eta}}(t \fr s)$ for every $s \in T$ with $o(s) < \bar{\eta}$.\vskip\bigskipamount

\noindent
Recall that for each sequence $s \in T$, $o(s) < \bar{\eta}$, the function $\pi^{t \fr s}(\nu)$ is a normal projection 
of $U_{\kappa,\bar{\eta}}(t \fr s)$ to $U_{\kappa,0}(t \fr s)$. Since  $Q^\nu(\emptyset)\uhr \lambda \times \pi^{t \fr s}(\nu)$ is bounded in $\pi^{t \fr s}(\nu)$,\footnote{i.e., $Q^\nu(\emptyset)\uhr \lambda \times \pi^{t \fr s}(\nu)\in V_{\pi^{t \fr s}(\nu)}$ and $\pi^{t \fr s}(\nu)$ is an inaccessible cardinal} we can press down on its value, and find a subset $A'_{s}(\bar{\eta}) \in U_{\kappa,\bar{\eta}}(t \fr s)$ of $A'(\bar{\eta})$, and a collapse condition $Q'(s)$ such that $Q^\nu(\emptyset)\uhr \pi^{t \fr s}(\nu) = Q'(s)$ for all $\nu \in A'_{s}(\bar{\eta})$.\vskip\bigskipamount

\noindent
By applying the construction arguments of Remark \ref{RMK:Qconstruction}, we may find direct extension $(t,T')$ of $(t,T)$ having both properties from parts (1), (2) of the remark,
 where (1) is applied with respect to the sets $A'_s(\bar{\eta})$, $s \in T'$, $o(s) < \bar{\eta}$, given by the pressing down 
 process above, by which $Q'(s)$ is defined.
 We note that, as mentioned at the end of part (2) of the remark, for every $s' \in T'(<\bar{\eta}) = \{ s \in T' \mid o(s) < \bar{\eta}\}$ which end extends $s$,  
\[\succ_{T'}(s) \cap \{ \nu \mid o(\nu) = \bar{\eta} \} \in U_{\kappa,\bar{\eta}}(t \fr s').\]

Moreover, since 
\[\succ_{T'}(s) \cap \{ \nu \mid o(\nu) = \bar{\eta}\} \subseteq A'_{s}(\bar{\eta}),\] 
we conclude that $A'_{s}(\bar{\eta}) \in U_{\kappa,\bar{\eta}}(s')$. It follows that  
$A'_s(\bar{\eta}) \cap A'_{s'}(\bar{\eta}) \neq\emptyset$, which in turn, implies that 
\[Q'(s) = Q'(s')\uhr \lambda \times \pi^{t \fr s}(s')\] 
(as witnessed by $Q^\nu(\emptyset)$ for any $\nu \in A'_s(\bar{\eta}) \cap A'_{s'}(\bar{\eta})$).\vskip\bigskipamount

Finally, we form a sub-tree $T''$ of $T'$ by intersecting $T'_{\l \nu \r}$ with $T^\nu$, for each $\nu \in \succ_{T'}(\emptyset) \cap \{\nu \mid o(\nu) = \bar{\eta}\} \subseteq A'(\bar{\eta})$. By appealing to part (3) of the previous remark, we can find a direct extension $(t,T^*)$ of $(t,T'')$ which further satisfies that 
 for every $s \in T^*$ for which $\nu_{s} := \pi^{t}_{\eta}(s) \in b_{s}\setminus b_t$ is defined,  $\nu_{s} \in \succ_{T^*}(\emptyset) \cap \{\nu \mid o(\nu) = \eta\}$ and there exists some $s'  \in T^*$ which end extends $\l \nu_s\r$,  such that $b_{t \fr s'} = b_{t \fr s}$ and  $T^*_{s} \subseteq T^*_{s'}$.
 Let $\bar{s} = s'\setminus \nu_s+1$.
 
$Q^{\nu_s}(\bar{s})$ is defined, since $T^*_{\l \nu_{s}\r} \subseteq T^{\nu_{s}} = \dom(Q^{\nu_{s}})$. Moreover, 
 since $(t\cup\{\nu_s\},T^{\nu_s},Q^{\nu_s})$ is assumed to be a condition in $\qo^*_{\kappa,\tau}$, the value
 $Q^{\nu_s}(\bar{s})$ does not depend on the choice of a sequence $s'$, and its associated sub-sequence $\bar{s} \in T^{\nu_s}$ satisfying $b_{t \fr\l \nu_s\r \fr \bar{s}} = b_{t \fr s'} = b_{t \fr s}$. \vskip\bigskipamount

We turn to define the function $Q^*$ on $T^*$. We follow the convention from the last paragraph, where for $s \in T^*$ with $o(s) \geq \bar{\eta}$, we denote $\nu_s = \pi^t_{\bar{\eta}}(s)$. We set
 \[
  Q^*(s) = 
  \begin{cases}
   Q^{\nu_s}(\bar{s}) &\mbox{ if }  o(s) \geq \bar{\eta}, \bar{s} \in T^{\nu_s} \text{ satisfies } b_{t \fr \l \nu_s \r \fr \bar{s}} = b_{t \fr s}\\
   Q'(s) &\mbox{ if } o(s) < \bar{\eta}
  \end{cases}
 \]
 
We claim that $Q^*(\emptyset)$ forces that $(t,T^*,Q^*)$ is a condition that extends $(t,T,Q)$.
We show first that $Q^*(\emptyset)$ forces $(t,T^*,Q^*)$ is a condition of $\qo^*_{\kappa,\tau}$, which requires verifying the first three conditions 
in the definition of the poset. Conditions (i), (ii) are clearly satisfied as $(t,T^*) \in \qo_{\kappa,\tau}$. 
To verify condition (iii), we need to check that for every $s,s' \in T^*$, if $s'$ extends $s$ then $Q^*(s) = Q^*(s')\uhr \lambda \times \pi^{t \fr s}(s')$.
The verification breaks down to {three cases}. \vskip\bigskipamount

\noindent
\textbf{Case I:} If $o(s),o(s') < \bar{\eta}$, then $s,s' \in T^*(<\bar{\eta})$ and, as described above, the result is an immediate consequence of the fact $s' \in T'(<\bar{\eta})$ end extends $s$. \vskip\bigskipamount

\noindent
\textbf{Case II:}
If $o(s) < \bar{\eta}$ and $o(s') \geq \bar{\eta}$ then $\nu' = \pi^{t\fr s}_{\bar\eta}(s') \in A'_s(\bar{\eta})$. As 
 $\pi^{t \fr s}(s') = \pi^{t \fr s}(\nu')$ and $(t \cup \{\nu'\}, T^\nu,Q^\nu) \in \qo^*_{\kappa,\tau}$, 
it follows that 
\[
\begin{matrix}
Q^*(s')\uhr \lambda \times \pi^{t \fr s}(s') = & \\
Q^{\nu'}(\bar{s'}) \uhr \lambda \times \pi^{t \fr s}(s') = &  \\
Q^{\nu'}(\emptyset) \uhr \lambda \times \pi^{t \fr s}(\nu') = & Q'(s) = Q^*(s) 
\end{matrix}\]

\noindent
\textbf{Case III:} If $o(s) \geq \bar{\eta}$ then there exists some $\bar{s} \in T^*_{\nu_{s}} \subseteq T^{\nu_s}$
such that $b_{t \fr \l \nu_s \r \fr \bar{s}} = b_{t \fr s}$ and $T^*_s \subseteq T^*_{\l \nu_s \r \fr \bar{s}}$. In particular, $Q^*(s) = Q^*(\bar{s})$ and $s' \in T^*_{\l \nu_s \r \fr \bar{s}} \subseteq T^{\nu_s}_{\bar{s}}$ . Since  $Q^{\nu_s}$ is $(t \cup \{\nu_s\}, T^{\nu_s})$-coherent, we conclude that 

\[ 
\begin{matrix}
Q^*(s')\uhr \lambda \times \pi^{t \fr s}(s') = & \\
Q^{\nu_s}(s') \uhr \lambda \times \pi^{t \fr s}(s') = & 
Q^{\nu_{s}}(\bar{s}) = Q^*(s)
\end{matrix}
\]

This concludes the proof that $p^* = (t,T^*,Q^*)$ satisfies the  property (iii) of the definition of $\qo^*_{\kappa,\tau}$, and thus, that $Q^*(\emptyset) \in \col(\lambda,<\kappa)$ forces it is a condition of $\qo^*_{\kappa,\tau}$. 
It is immediate from its definition that $p^*$  is a direct extension of $p$.
Finally, our choice of the tree $T^*$  above, obtain from $T'$ using fact (3) from Remark \ref{RMK:Qconstruction} above, implies at once that 
 $Q^*(\emptyset)$ forces 
 $\{ (t\fr \l \nu \r,T^\nu,Q^\nu) \mid \nu \in A'(\bar{\eta}) \}$ to be predense above $p^*$. 
\end{proof}

\begin{lemma}
$\qo^*_{\kappa,\tau}$ satisfies the Prikry Property.
\end{lemma}
\begin{proof}
 Suppose otherwise. Working in $V[G \ast H]$, let $(t,T,Q)$, $\sigma$ be condition and statement of $\qo^*_{\kappa,\tau}$, respectively, such that no direct extension of $(t,T,Q)$ decides $\sigma$. 
 Back in $V[G]$, let $q \in H$ be a condition which forces this statement about $(t,T,Q)$ and $\sigma$.
 Since $q$ forces $(t,T,Q)$ to be a condition of $\qo^*_{\kappa,\tau}$ we have that $q \geq Q(\emptyset)$. Therefore, by moving to a
 direct extension of $(t,T,Q)$, we may assume that $q = Q(\emptyset)$.
 For notational simplicity, we make the assumption that $t = \emptyset$. The proof for an arbitrary sequence $t$ is similar.
 
 Let $A = \suc_T(\emptyset)$. We may assume that $q\in V_{\mu_0}$, where $\mu_0 = \min(\{ \pi_0^\emptyset(\nu) \mid \nu \in A\})$. 
 For each $\nu \in A$, we choose a condition $q(\nu) \in \col(\lambda,<\kappa)$, extending $q \cup Q(\nu)$, which decides the $\col(\lambda,<\kappa)$ statement of whether there exists a direct extension 
 $p^{\nu} = (\l \nu \r,T^\nu,Q^\nu)$ 
 of $p \fr \l \nu \r$ which decides $\sigma$, and if so, whether $p^\nu$ forces $\sigma$ or $\neg\sigma$.
 Let $A^0$ be the sets of $\nu \in A$ for which $q(\nu)$ forces $p^\nu$ exists, and  ``$p^\nu \Vdash \sigma$''.
 Similarly, let $A^1 \subseteq A$ consists of $\nu$ such that $q(\nu)$ forces $p^\nu$ exists, and  ``$p^\nu \Vdash \neg\sigma$'',
 and $A^2 = A \setminus (A^0 \uplus A^1)$. The proof splits now into three main cases: \vskip\bigskipamount
 
 %\ref{Lem.amalgamation.in.Q*})
 \noindent
 \textbf{Case 0:} There exists some $\bar{\eta} < \tau$ such that $A^0 \in U_{\kappa,\bar{\eta}}(\emptyset)$.\vskip\bigskipamount
 Let $A'(\bar{\eta}) = A^0 \cap \{ \nu \mid o(\nu) = \bar{\eta}\}$. By applying Lemma \ref{Lem.amalgamation.in.Q*} with respect to the family of conditions 
 $\{ p^{\nu} =( \l \nu \r,T^\nu,Q^\nu) \mid \nu \in A'(\bar{\eta}) \}$, we can find a 
 $p^* = (\emptyset,T^*,Q^*)$ which is forced by $Q^*(\emptyset)$ to be a direct extension of $(\emptyset,T,Q)$, and to have
  $\{ p^{\nu} \mid \nu \in A'(\bar{\eta}) \}$ be a predense in $\qo^*_{\kappa,\tau}/(\emptyset,T^*,Q^*)$.
 It follows that  $Q^*(\emptyset) \geq Q(\emptyset) = q$ forces $p^* = (\emptyset,T^*,Q^*)$ is a direct extension of $p$ which decides $\sigma$. Contradicting our assumption. \vskip\bigskipamount

 \noindent
 \textbf{Case 1:} There exists some $\bar{\eta} < \tau$ such that $A^1 \in U_{\kappa,\bar{\eta}}(\emptyset)$.\vskip\bigskipamount
 The argument for this case is similar to the previous one, and leads to an extension $q^* \geq q$ in $\col(\lambda,<\kappa)$, and a direct extension $\bar{p} \geq^* p$, such that $q^*$ forces
 $\bar{p} \Vdash_{\qo^*_{\kappa,\tau}}\neg\sigma$. Contradiction.\vskip\bigskipamount
 
 \noindent
\textbf{Case 2:} $A^2 \in \bigcap_{\eta < \tau} U_{\kappa,\eta}(\emptyset)$. \vskip\bigskipamount
\noindent
Let $A^2(0) = A^2 \cap \{ \nu \mid o(\nu) = 0\}$, and apply Lemma \ref{Lem.amalgamation.in.Q*} with respect to 
$A^2(0)$ and $\{ p \fr \l \nu \r \mid \nu \in A^2(0)\}$,  
to obtain a direct extension
$p_0^* = (\emptyset, T_0^*,Q_0^*)$ of $p$, with $\{ p\fr \l{\nu}\r \mid \nu \in A^2(0) \}$ being predense in $\qo^*_{\kappa,\tau}/p_0^*$.

Denoting $q_0^* = Q_0^*(\emptyset)$, we define $\bar{A}^2$ to be the set of all  $\nu \in A^2$ for which $q^*_0$ forces there is no direct extension $p^\nu \geq^* p' \fr \nu$ which decides $\sigma$. 
Note that $\bar{A}^2$ must belong to $\bigcap_{\eta < \tau}U_{\kappa,\eta}(\emptyset)$, since otherwise, there would be some $\eta < \tau$,  and a set $A'(\eta) \subseteq \bar{A}^2$ consisting of $\nu$ for which some $q^*(\nu) \geq q^*_0$ forces there exists a direct extension of $p^*_0 \fr \nu$ which decides $\sigma$. This, in turn would allow us to repeat the construction of one of the previous cases 0 and 1, to show that there is  $q^* \geq^* q^*_0$ which forces some direct extension $p^*$ of $p^*_0$ to force either $\sigma$ or $\neg\sigma$, contradicting the choice of $q^*_0$.

Let $q^1:=q^*_0$ and $p^1 = ( \emptyset, T^1, Q^1)$ be the direct extension of $p^*_0$ obtained by shrinking $\suc_{T^{p_0^*}}(\emptyset)$ to points in $\bar{A}^2$. It follows from the construction $q^1$ forces that for all  $\nu \in \suc_{T^1}(\emptyset)$, $p^1 \fr \l \nu \r$ does not have a direct extension in $\qo^*_{\kappa,\tau}$ which decides $\sigma$.

Next, we move up to the second level of the tree. To each $\mu \in \suc_{T^1}(\emptyset)$, we can repeat the above analysis with respect to $q(\mu) = q^1 \cup Q^1(\mu)$ and $p^1 \fr \l\mu\r = ( \l \mu \r, T^1_{\l \mu \r}, Q^1_{\l \mu \r})$. Accordingly, we  split $B = \suc_{T^1}(\emptyset)$ into three sets, $B^0, B^1, B^2$, based on whether the analysis for
$q(\mu)$ and $p^1 \fr \l\mu\r$ has produced an extension $q^*(\mu) \geq q(\mu)$ which forces some direct extension $p^{1,\mu} \geq^* p^1 \fr \l\mu\r$ to decide $\sigma$ ($B^0$ and $B^1$ for forcing $\sigma$ and $\neg\sigma$, respectively), or not. 
The argument above shows that if $B^0$ or $B^1$ belong to $U_{\kappa,\eta}(\emptyset)$ for some $\eta < \tau$ then there exists some $q^* \geq q^1$ which forces that $p^1$ has a direct extension which decides $\sigma$, contradicting our assumptions. 

It follows that $B^2 \in \bigcap_{\eta < \tau}U_{\kappa,\eta}(\emptyset)$, and by repeating the argument from the beginning of Case 2 for each $p^1\fr \l\mu\r$, $\mu \in B^2$, we can find for each $\mu \in B^2$,  conditions $q^*(\mu) \geq q^1(\mu)$, $q^*(\mu) \in \col(\lambda,\pi^\emptyset_0(\mu))$, and  $p^*_\mu = ( \l \mu \r, T^*_\mu, Q^1_\mu) \geq^* p^1 \fr \mu$, such that $q^*(\mu)$ forces there is no direct extension of $p^*_\mu \fr \nu$ which decides $\sigma$, for any $\nu \in \suc_{T^*_\mu}(\emptyset)$.  
We may assume $q^*(\mu) = Q^1_\mu(\emptyset)$ and apply Lemma \ref{Lem.amalgamation.in.Q*} with respect to 
$B^2(0)$ and $\{ p^*_{\mu}  \mid \mu \in B^2(0)\}$, to conclude, similarly to the above, that there are extensions $q^2 \geq q^1$ and  $p^2 = (\emptyset,T^2,Q^2) \geq^* p^1$, such that $q^2$ forces that for any pair $\l \nu_0,\nu_1\r \in T^2$, $p^2_{\l \nu_0,\nu_1 \r}$ does not have a direct extension which decides $\sigma$.

The construction is now repeated level by level,  for all $n < \omega$. This produces sequences of extensions $q = q^0 \leq q^1 \leq \cdots \leq q^n \cdots$ in $\col(\lambda,<\kappa)$ and $p = p^0 \leq^* p^1 \leq^* \cdots \leq^* p^n \cdots$ in $\qo^*_{\kappa,\tau}$, such that for each $n < \omega$, writing $p^n = (\emptyset, T^n,Q^n)$, we have that 
$q^n$ forces that for all sequences $s \in T^n$ of length $|s| \leq n$, $p^n \fr s$ does not have a direct extension which decides $\sigma$. 
Finally, let $q^\omega \in \col(\lambda,<\kappa)$ be an union of all $q^n$, $n < \omega$, and $p^* \in \qo^*_{\kappa,\tau}$ be a direct extension of $p^n$ for all $n < \omega$. Writing $p^* = (\emptyset, T^*,Q^*)$, it follows from the construction that $p^\omega$ forces that for no $s \in T^*$ such that $p^* \fr s$ has a direct extension which decides $\sigma$. This is of course absurd.
\end{proof}

We conclude that $( \qo^*_{\kappa,\tau},\leq,\leq^*)$ is a Prikry type forcing whose direct extension order $\leq^*$ is $\lambda$-closed. In particular, it does not add bounded subsets to $\lambda$. 
Moreover, like $\qo^\tau_\kappa$,  $(\qo^*_{\kappa,\tau},\leq)$ introduces a generic club $b^\tau_\kappa \subseteq \kappa$ of order-type $\otp(b^\tau_\kappa) = \omega^\tau$. 
Finally,  since   $U_{\kappa,\tau'} \subseteq F_{\kappa,\tau'}(t)$ is clearly contained in $D_{\kappa,\tau',t}$ for every coherent sequence $t$, 
it follows from a standard density argument that if $b_{\kappa}^\tau$ is a $\qo^*_{\kappa,\tau}$-generic sequence over $V[G\ast H]$, then for every $V$-set $A \in \F_\kappa$ there exists some $\beta < \kappa$ such that $b^\tau_\kappa \setminus \beta \subseteq A$.

\subsection{The forcing $\bar{\co}_{\F_{\kappa}}$}\label{subsection:c-star}
Our goal now is to introduce a poset $\bar{\co}_{\F_\kappa}$ which is equivalent to $\co_{\F_\kappa}$, and further has a dense sub-forcing $\co^*_{\F_\kappa}$, which is of Prikry-type, and its direct extension order is $\lambda$-closed. 
We first introduce the poset $\co^*_{\F_\kappa}$, obtained from $\qo^*_{\kappa,\lambda}$ and $\co_{\F_\kappa}$. %\vskip\bigskipamount

\noindent
Recall that $G \ast H$ is $V$ generic for $\po \ast \col(\lambda,<\kappa)$. 
Working in $V[G \ast H]$ we consider the two-step iterations
$\qo^*_{\kappa,\lambda} \ast \co_{\F_\kappa}$ consisting of conditions $(q, \name{x})$ so that $q \force_{\qo^*_{\kappa,\lambda}} \name{x} \in \check{\co_{\F_\kappa}}$.
Note that when forcing with $\co_{\F_\kappa}$ over a $V[G \ast H]$-generic extension by $\qo^*_{\kappa,\lambda}$, we require that the bounded closed sets $c \subseteq\kappa$ in the conditions
$x = \l c,A\r \in \co_{\F_\kappa}$ are actually ground model sets, from $V[G \ast H]$. 
In particular, for every condition $(q,\name{x})$ there exists an extension $q' \geq q$ and a pair $x \in \co_{\F_\kappa}$ so that $q' \Vdash \name{x} = \check{x}$. \vskip\bigskipamount

Let $b^\lambda_\kappa \subseteq \kappa$ be a $\qo^*_{\kappa,\lambda}$ generic club in $\kappa$. We know that $\otp(b^\lambda_\kappa) = \lambda$ and that $b^\lambda_\kappa$ is almost contained in every set $A \in \F_{\kappa}$. 
Working in a $\qo_{\kappa,\lambda}^*$ generic extension $V[G \ast H \ast b^\lambda_\kappa]$ of $V[G \ast H]$ we see that for every condition $x = \l c, A\r$ in $\co_{\F_\kappa}$ there exists some $\beta \in A \setminus (\max c + 1)$ such that $x' = \l c', A\r$, with $c' = c \cup \{\beta\}$, extends $x$ and satisfies that $b^\lambda_\kappa \setminus (\max c'+1) \subseteq A$.

\begin{definition}[$\co^\lambda_{\F_\kappa}$] ${}$\vskip\bigskipamount
Working in a $\qo_{\kappa,\lambda}^*$ generic extension $V[G \ast H \ast b^\lambda_\kappa]$ of $V[G \ast H]$,  
let $\co^\lambda_{\F_\kappa}$ denote the subset of $\co_{\F_\kappa}$, consisting of conditions $x' = \l c', A'\r$ so that 
$\max(c') \in b^\lambda_\kappa$, and $b^\lambda_\kappa \setminus (\max(c')+1) \subseteq A'$. 
\end{definition}

It follows from the above that $\co^\lambda_{\F_\kappa}$ is a dense subset of $\co_{\F_\kappa}$. 
Since $b^\lambda_\kappa \subseteq\kappa$ is closed of order-type $\lambda = \cf(\kappa)^{V[G \ast H \ast b^\lambda_\kappa]}$, and no sequences of ordinals of length $<\lambda$ are introduced by $b^\lambda_\kappa$, it follows that the restriction of the $\co_{\F_\kappa}$ order to $\co^\lambda_{\F_\kappa}$ is $\lambda$-closed. \vskip\bigskipamount

With this observation, we move back to $V[G \ast H]$ to define the poset $\co^*_{\F_\kappa}$. 
\begin{definition}[$\co^*_{\F_\kappa}$]${}$\vskip\smallskipamount
 Let $\co^*_{\F_\kappa}$ be the two step iteration $\co^*_{\F_\kappa} = \qo^*_{\kappa,\lambda} \ast \co^\lambda_{\F_\kappa}$.
We define the direct extension ordering $\leq^*$ of $\co^*_{\F_\kappa}$ to be the extension of the usual direct extension order of $\qo^\lambda_\kappa$ with the standard order on the
second $\co^\lambda_{\F_\kappa}$ component.
\end{definition}
\begin{corollary}\label{cor:bar-C-prikry-type}
$\co^*_{\F_\kappa}$ is a dense sub-forcing of $\qo^*_{\kappa,\lambda} \ast \co_{\F_\kappa}$ which satisfies the Prikry Property, and its direct extension order is $\lambda$-closed. 
\end{corollary}

Note that for every dense subset $D$ of $\co_{\F_\kappa}$ and a condition $( q, \name{x}) \in \co^*_{\F_\kappa}$ there exists a direct extension $(q^*,\name{x}^*) \geq^* (q,\name{x})$ such that $q^* \Vdash \name{x}^* \in \check{D}$. \vskip\bigskipamount
Similarly, it is clear that the set of conditions $(q',\check{x'}) \in \co^*_{\F_\kappa} $, for which the second component is a canonical name $\check{x'}$ of a condition $x' \in \co_{\F_\kappa}$, is dense in $\co^*_{\F_\kappa}$. 
The map  $(q',\check{x'})  \mapsto x'$ defined on this dense set naturally induces a forcing projection $\pi$ from $\co^*_{\F_\kappa}$ to the boolean completion of $\co_{\F_\kappa}$. This projection sends a condition of the form $\langle q, \name{x}\rangle$ to the join of the collection of all $y \in \co_{\F_\kappa}$ such that there is some extension of $q$, $q'$ that forces $\name{x} =\check{y}$.

\vskip\bigskipamount

Next, we follow Gitik's machinery from \cite{Gitik-HB}, to form  a Prikry-type forcing notion $\bar{\co}_{\F_\kappa}$ which is equivalent to $\co_{\F_\kappa}$, from $\co^*_{\F_\kappa}$.

\begin{definition}[$\bar{\co}_{\F_\kappa}$]${}$\vskip\smallskipamount
We define a Prikry-type forcing notion  $(\bar{\co}_{\F_\kappa},\leq',\leq^*)$ as follows. 
\begin{itemize}
 \item 
$\bar{\co}_{\F_\kappa}= \co^*_{\F_\kappa}$,  
\item the partial ordering $\leq'$ is defined by 
$p' \geq' p$ if $\pi(p') \geq \pi(p)$, and 
\item $\leq^*$ is taken to be the same direct extension order of $\co^*_{\F_\kappa}$
\end{itemize}
\end{definition}

It is immediate from the definition that $(\bar{\co}_{\F_\kappa}, \leq')$ is equivalent as a forcing notion to $(\co_{\F_\kappa},\leq)$ and that the direct extension order $\leq^*$ of $\bar{\co}_{\F_\kappa}$ is $\lambda$-closed. \vskip\bigskipamount
To show that $(\bar{\co}_{\F_\kappa},\leq',\leq^*)$ satisfies the Prikry Property, it suffices to verify that 
for every statement $\sigma$ in the forcing language of $\co_{\F_\kappa}$ and every condition $p \in \bar{\co}_{\F_\kappa}$, there is a direct extension $p^* \geq^* p$ such that $\pi(p^*)$ decides $\sigma$. 
Indeed, defining $D_0 = \{ p' \in \co^*_{\F_\kappa} \mid \pi(p') \Vdash \sigma\}$ and $D_1 = \{ p' \in \co^*_{\F_\kappa} \mid \pi(p') \Vdash \neg\sigma\}$, it is clear that $D_0 \cup D_1$ is dense in $\co^*_{\F_\kappa}$ and that a generic filter $G^*$ of $\co^*_{\F_\kappa}$ will have a nontrivial intersection with exactly one of the two sets. Let $\sigma^*\colon \name{G}^* \cap D_0 \neq\emptyset$. Then $\sigma^*$ is a statement for the forcing language of $\co^*_{\F_\kappa}$. Moreover,  it is clear from our construction that for a condition $p^* \in \co^*_{\F_\kappa}$ which decides $\sigma^*$ we have that  $p^* \Vdash \sigma^*$ implies that $\pi(p^*) \Vdash \sigma$, and 
$p^* \Vdash \neg\sigma^*$ implies that $\pi(p^*)\Vdash \sigma$. Since $\co^*_{\F_\kappa}$ satisfies the Prikry Property every condition $p$ has a direct extension $p^* \geq^* p$ which decides $\sigma^*$. \vskip\bigskipamount

%\begin{remark}\label{RMK:C*homog}
 We note that similarly to $\po \ast \col(\lambda,<\kappa) \ast  \co_{\F_\kappa}$, the forcing $\po \ast \col(\lambda,<\kappa)  \ast \co^*_{\F_\kappa}$ is cone homogeneous. In most applications of homogeneity, moving to an equivalent forcing does not change the main properties of its iterations.
 In order 
 to apply the results from section \ref{Sec:homog} and argue that the iteration 
 $\po \ast \col(\lambda,<\kappa) \ast \bar{\co}_{\F_\kappa}$
 is cone homogeneous, we need to verify that this posets meets the assumptions of Lemma \ref{FACT:homogiter}. 
 
 %However, a word of caution is in order here since this assertion does not suffice for concluding that a \emph{Prikry-type iteration framework} of our posets is also weakly homogeneous. The potential issues lies in the restriction of the iteration conditions $\vec{q} = \l q_\alpha \mid \alpha < \theta\r$ to satisfy that 
 %$q_\alpha$ is a direct extension of the trivial condition for all but finitely many ordinals $\alpha < \theta$.
 %Therefore, if there are infinitely many (names) of  automorphisms $\sigma_\alpha$ which for which $\sigma_\alpha(q_\alpha)$ is not a direct extension of the trivial condition, although $q_\alpha$ is, then it is impossible to combine them to form automorphisms of the Prikry-type iteration.  
 %Thus, we will need a slightly stronger type of (weak) homogeneity property for our purposes. Such property is established in the next lemma. 
 
\begin{lemma}\label{Lem:Q[U].weakly.homogeneous}
Denote $\po \ast \col(\lambda,<\kappa)  \ast \co^*_{\F_\kappa}$ by $\wo$, and its regular and direct extension orders by $\leq_{\wo}$ and $\leq^*_{\wo}$ respectively.\\
For every $w_0, w_1 \in \wo$ there are direct extensions $w_0^*,w_1^*$ of $w_0,w_1$ respectively, 
and a cone isomorphism $\varphi \colon \wo/w_0^* \to \wo/w_1^*$ which respects the direct extension order $\leq^*_{\wo}$.

In particular, the forcing $\po \ast \col(\lambda,<\kappa)\ast \bar{\co}_{\F_\kappa}$ is cone homogeneous.
\end{lemma} 

We observe that assuming the coherent sequence $\U$  (by which $\wo = \po \ast \col(\lambda,<\kappa)  \ast \co^*_{\F_\kappa}$ is defined) is ordinal definable in $V$, 
then the statement of the lemma guarantees that $\wo$ satisfies the requirements of the iterated poset $\qo_\alpha$ from Lemma \ref{FACT:homogiter}. 
\begin{proof}
Let us start with the last assertion. Since the identity is a projection from $\co^*_{\F_\kappa}$ to $\bar{\co}_{\F_\kappa}$, an isomorphism of a cone of elements in $\co^*_{\F_\kappa}$ naturally induces an isomorphism of the corresponding cone in $\bar\co_{\F_\kappa}$.

Let $w_0 = \langle p_0, c_0, \langle q_0, \name{x}_0\rangle\rangle$ and 
$w_1 = \langle p_1, c_1, \langle q_1, \name{x}_1\rangle\rangle$, where the conditions $q_0 = \langle t_0, T_0, Q_0\rangle$ and $q_1 = \langle t_1, T_1, Q_1\rangle$ belong to $\qo^*_{\kappa,\lambda}$. 

By \cite[Theorem 4.6]{bunger} applied to the iteration $\po \ast \qo_{\kappa,\lambda}$, there are direct extensions $\l p_0^*, (t_0,T) \r$ of 
$\l p_0, (t_0,T_0) \r$, and $\l p_1^*, (t_1,T) \r$ of $\l p_1, (t_1,T_1) \r$, with a common top tree $T$, and an 
isomorphism $\psi$ between the cone below
$\l p_0^*, (t_0,T) \r$ and the cone below
$\l p_1^*, (t_1,T) \r$.
We record here that the map $\psi$ constructed in the proof of \cite[Theorem 4.6]{bunger} satisfies two additional 
 properties: First, it does not make any changes to the $\F_{\kappa}$-trees $S$ appearing in conditions $p \fr \l s,S \r \in \po \ast \qo_{\kappa,\lambda}$.
 %(i.e.,  to $\l t_1 \fr s, S\r$, i.e., without modifying the tree $S$.
Second, it respects the direct extension order of  $\po \ast \qo_{\kappa,\lambda}$.  \vskip\bigskipamount

Next, we move to examine the Levy collapse condition and the suitable functions in the conditions from $\wo$. 
Since the collapsing forcing $\col(\lambda, <\kappa)$ is evaluated in the generic extension by $\po$, $\psi$ naturally acts also on the $\po$-names $c_0$ and $Q_0$ that appear in $w_0$. As usual, we denote the resulting names by $c_0^\psi$ and $Q_0^\psi$. 
Let $\tau_\emptyset$ be a $\po$-name of an automorphism of the Levy collapse poset which maps an extension $c_0'$ of $c_0^{\psi}$ to an extension $c_1^*$ of $c_1$, 
and define $c_0^* = (c_0')^{\psi^{-1}}$. 
Note that since $c_i$ forces $Q_i(\emptyset) \in \name{H}$, we may extend $Q_0(\emptyset),Q_1(\emptyset)$ to $Q_0^*(\emptyset),Q_1^*(\emptyset)$ so that  $c_i = Q_i^*(\emptyset)$ for $i = 0,1$. 
Next, for each $s \in T$ and $i = 0,1$, define $\dom_1(Q_i(s)) = \{ \alpha < \kappa \mid Q_i(s)\uhr \lambda \times \{\alpha\} \neq \emptyset\}$
and $\rho^i_s = \sup(\dom_1(Q_i(s)))$. Set $\rho_s = \max(\rho_s^0,\rho_s^1)$. 
By moving to a direct extension tree $T^*$ of $T$, we may assume that 
for every $s \in T^*$ and $\nu \in \succ_{T^*}(s)$, $\pi^{t \fr s}_0(\nu) > \rho_s$ where the projection is computed in both generic extensions. 
This leaves enough space between the conditions $Q_i(s), Q_i(s \fr \l \nu \r)$, $i = 0,1$, to define autormorphisms taking an extension $Q_0^*(s)$ of $Q_0(s)$ to an extension $Q_1^*(s)$ of $Q_1(s)$, without conflicting with $Q_i(s \fr \l \nu \r)$, $i = 0,1$. We  can therefore define by induction on the lexicographic order $<_{lex}$ on $T^*$ (where two sequences are compared from their top elements down)
 automorphisms $\tau_s$, $s \in T^*$, of $\col(\lambda,<\kappa)$, and collapse extensions $Q_i^*(s) \geq Q_i(s)$, $i = 0,1$, 
with the following properties: For all $s \in T^*$, 

\begin{itemize}
\item $\tau_s$ is supported in $\col(\lambda,<\rho_s)$,\footnote{I.e., for every $p \in \col(\lambda,<\kappa)$, if $p = p_0 \cup p_1$ where $p_0 = p \uhr \lambda \times \rho_s$, then $\tau_s(p) = \tau_s(p_0) \cup p_1$.}
\item $\tau_s((Q_0^*)^\psi(s)) =Q^*_1(s)$,
\item If $s' \in T^*_s$ then $\tau_{s'}\uhr \col(\lambda,<\pi_0^{t \fr s}(s)) = \tau_s$,
\item If $\bar{s} \in T^*$ and $b_{t_0 \fr s} = b_{ t_0 \fr \bar{s}}$ then $\tau_{s} = \tau_{\bar{s}}$. 
\end{itemize}

%The conditions above imply that the permutation that induces $\tau_{s \fr \l \nu \r}$ is a composition of two permutations with disjoint domains: the one that induces $\tau_s$ and a permutation of $[\rho_s, \rho_{s \fr \l\nu\r})$. Recall that for every $s$ in the tree and for every $\nu$ in the tree successor of $s$, $\pi^{t_1 \fr s}(\nu) > \rho_s$. Moreover, we may assume that for all $s$ in the tree and $\nu \in s$, if $\nu' = \pi^\nu(s)$ then $\nu' > \rho_{s \cap \nu}$. Thus, we can define an automorphism $\tau_{b_s}$ by glueing together the permutations that induces $\tau_{s'}$ for $s' \subseteq b_s \setminus b_t$. Since those permutations agree on their common domain, this is possible.

Let $\bar{\wo} = \po \ast \col(\lambda,<\kappa) \ast \qo_{\kappa,\lambda}^*$  be the initial forcing iteration of 
\[\wo = \po \ast \col(\lambda,<\kappa) \ast \co^*_{\F_\kappa} = \po \ast \col(\lambda,<\kappa) \ast (\qo_{\kappa,\lambda}^* \ast \co^\lambda_\kappa).\]
Let $w_0\uhr \bar{\wo} = \l p_0, c_0, \l t_0,T_0,Q_0\r \r$ and $w_1\uhr \bar{\wo} = \l p_1,c_1,\l t_1,T_1,Q_1\r \r$ be the restrictions of $w_0,w_1$ to $\bar{\wo}$
and consider their direct extensions $\bar{w_o},\bar{w_1}$ in $\bar{\wo}$, defined by 
$\bar{w_i} = \l p^*_0, c^*_0, \l t_0,T^*,Q^*_0\r \r$, $i = 0,1$. \vskip\bigskipamount

Our choice of cone isomorphism $\psi$ for $\po \ast \qo_{\kappa,\lambda}$ together with 
the collection of Levy-Collapse automorphisms $\vec{\tau} = \{\tau_s\}_{s \in T^*}$, naturally induces a function $\bar{\varphi}$ on the cone $\bar{\wo}/ \bar{w_0}$, defined as follows.
For a condition
$\bar{w} = \l p, c,\l s , S, Q\r\r $, we set  $\bar{\varphi}(\bar{w}) = \l p',c',\l s', S,Q'\r\r$ to be: 
\[
\l p',\l s',S\r\r = \psi(\l p',\l s,S\r\r ) , \quad c' = \tau_{s}(c^\psi), \quad  Q'(s') = \tau_{s \fr s'}(Q(s'))
\]

We claim that $\l p',c',\l s',S, Q'\r\r$ is a condition in $\bar{\wo}$.
First, it is immediate from our choice of $p,c$ that  $\l p',c'\r \in \po * \col(\lambda,<\kappa)$. It therefore remains to verify that 
$\l s',S,Q'\r$ is forced by $\l p',c'\r$ to be a condition in $\qo^*_{\lambda,<\kappa}$.
The fact that $\l s',S,Q'\r$ satisfies requirements $(1)$ and $(2)$ of definition \ref{def:Q*poset} is immediate. 
To verify  the coherency requirement $(3)$ of definition \ref{def:Q*poset}, we note that $Q'$ is forced by $p'$ to be $(s, S)$-suitable. 
Indeed, it follows from our choice of $\tau_{s'}$ that its support is bounded below the projection $\pi^{s \fr s'}_0(\nu)$, for any $\nu \in \succ_{T^*}(s')$.
Property $(4)$ follows from the fact that the statement ``$c \Vdash_{\col(\lambda,<\kappa)} Q(\emptyset) \in \name{H}$'' is forced by $p \in \po$, 
which implies that $\tau_\emptyset(c^\psi) \Vdash \tau_\emptyset(Q(\emptyset)^\psi) \in \tau_\emptyset(\name{H}^\psi)$ is forced by $p'$. This, combined with definition of 
$c'$ and $Q'$, and the fact that the name $\name{H}$ is a fixed point of both $\tau_\emptyset$ and $\psi$, guarantees that requirement $(4)$ is satisfied.  
Next, $(s',S,Q')$ satisfies requirement $(5)$ of definition \ref{def:Q*poset} by a similar argument to the previous one, using the fact $(s, S, Q)$ satisfies property $(5)$ together with the last property listed above for $\{\tau_s\}_{s\in T}$.
Having verified that  $\l p',c',\l s',S, Q'\r\r$ is a condition in $\bar{\wo}$ it is straightforward to check that it extends $\bar{w_1}$ and 
thus $\bar{\varphi}\colon \bar{\wo}/\bar{w_0} \to \bar{\wo}/\bar{w_1}$ is a well defined function. In order to show that it is cone isomorphism we need to show that it is order-preserving. 

Let us remark that the automorphism $\psi$ modifies the values of $b_s$ for $s \in T^*$ by changing the value of their initial segments. Since those initial segments do not affect the definition of $\tau_{b_s}$, we will ignore it and write always $b_s$ instead $b_s^\psi$.  

Let $\bar{w}_1 = \l p_1, c_1, \l s_1, S_1, Q_1\r \r,\, \bar{w}_2 =\l p_2, c_2, \l s_2, S_2, Q_2\r \r$ be pair of conditions in the cone above $\bar{w}_0$. We need to show that $\bar{\varphi}(\bar{w}_1) \leq \bar{\varphi}(\bar{w}_2)$ if and only if $\bar{w}_1 \leq \bar{w}_2$. 

For direct extensions, this is clear, as the tree $S_1$ does not move under $\bar\varphi$. Let us assume that $\bar{w}_2$ is a one-point extension of $\bar{w}_1$, by the point $\l\nu\r$. By moving to a dense subset, $c_2 \geq c_1, Q_1(\l \nu \r)$ and $b_{s_2} = b_{s_1 \fr \l\nu\r}$. Let us apply $\bar{\varphi}$ on $\bar{w}_1, \bar{w}_2$. The trees $S_1$ and $S_2$ do not move, so we must verify that $\l\nu\r$ is still a legitimate choice for an one-point extension of $\bar\varphi(\bar{w}_1)$. Indeed, $\tau_{b_{s_2}}(c_2)$ is (by the definition of $\tau_{b_{s_2}}$) stronger than $\tau_{\l \nu\r}(Q_1(\l\nu\r)$. Thus, we conclude that $\bar\varphi(\bar w_2)$ is an one-point extension of $\bar\varphi(\bar w_2)$ by $\l\nu\r$. The other direction is the same.\vskip\bigskipamount

 Finally, to obtain a desirable cone isomorphism $\varphi$ for $\wo = \bar{\wo} \ast \co^\lambda_\kappa$, it remains to extend $\bar{\varphi}$ to the final additional components $\name{x_0},\name{x_1}$ of $\co^\lambda_{\F_\kappa}$. 
 The proof Lemma \ref{Lem:CHomog} shows that there are $\bar{\wo}$-names  $\name{y}_0',\name{y}_1$ of extensions of $\name{x}_0^{\bar{\varphi}},\name{x}_1$ respectively, 
 and a name of a cone isomorphism $\sigma\colon \co^\lambda_\kappa/ \name{y}_0' \to \co^\lambda_\kappa/\name{y}_1$. 
 Accordingly, we set $\name{y}_0 = (\name{y}_0')^{\bar{\varphi}^{-1}}$ and 
define direct extensions $w_0^* \geq^* w_0,w_1^*\geq^* w_1$ and a map $\varphi\colon \wo/w_0^* \to \wo/w_1^*$ by 
 $w_i^* = \bar{w}_i \fr \name{y}_i$ and 
 \[ \varphi( \l \bar{w},\name{y} \r) = \l \bar{\varphi}(\bar{w}),\sigma(\name{y}^{\bar{\varphi}})\r.\]
 The fact $w_0^*,w_1^*,\varphi$ satisfy the result stated in the lemma is an immediate consequence of the fact $\bar{w}_0,\bar{w}_1,\bar{\varphi}$ satisfy similar properties for $\bar{\wo}$ and our choice of $\varphi$,$\name{y}_0,\name{y}_1$. 
\end{proof}
\newpage

%\newpage
\section{Strong Measurability at Successors of Singulars}\label{Sec:suc.of.sing}
Suppose that $V = \HOD$, $\kappa$ is a supercompact cardinal and $\lambda > \kappa$ is a measurable cardinal with a normal measure $\mathcal{U}$.  
We would like to construct a cone homogeneous poset in $V$ which will collapse $\lambda$ to be the successor of $\kappa$, change the cofinality of $\kappa$ to $\omega$,  and add a closed unbounded subset 
of $\lambda$ whose restriction to the set of $\{ \alpha < \lambda \mid \alpha \text{ is regular in } V\}$ is almost contained in every set $A \in \mathcal{U}$.\vskip\bigskipamount

It is natural to attempt obtaining this result by starting with an indestructible supercompact cardinal $\kappa$, and forcing with a Levy collapse of $\lambda$ to $\kappa^+$ followed by a Prikry forcing at $\kappa$ and a club forcing at $\lambda$. The difficulty with this approach is in its second step, where the choice of the measure on $\kappa$ depends on the generic filter for the Levy collapse and might lead to a Prikry generic sequence which will introduce to $\HOD$ information about the collapse of $\lambda$ to $\kappa^+$, and in particular prevent from $\HOD$ to witness that $\lambda$ is a measurable cardinal. %\vskip\bigskipamount

Instead, our approach will be based on recent use of the supercompact extender based forcing, introduced by Merimovich (\cite{Merimovich3}). Given a supercompact cardinal $\kappa$, we derive a $(\kappa,\lambda)$-supercompact extender $E$ from a supercompact embedding $j\colon V \to M$ for which ${}^{\lambda}M \subseteq M$. Let $\mathbb{P}_{E}$ be the supercompact extender based forcing associated to the extender $E$ of \cite{Merimovich3}. The conditions of $\po_E$ are pairs of the form $\langle f, T\rangle$ where $f$ is roughly a condition in the Cohen forcing and $T$ is a tree, with large splittings. We denote by $\po_E^*$ the Cohen part. 

The forcing $\po_E$ preserves $\lambda$ and singularizes all the regular cardinals in the interval $[\kappa, \lambda)$. We will follow the definitions and notations of \cite{Merimovich3}. %\vskip\bigskipamount
In \cite{GitikMerimovich}, Gitik and Merimovich show that this forcing is weakly homogeneous (and therefore cone homogeneous). 

Let $\mathcal U$ be a normal measure on $\lambda$ in the ground model. We would like to force a club to diagonalize $\mathcal{U}$ relative to the set of $V$-regular cardinals below $\lambda$. 
Note that the ordinals of uncountable cofinality below $\lambda$ in the extender based forcing extension are of measure zero in $\mathcal{U}$. Therefore, our club shooting poset has to allow $V$-singular ordinals as well. 
Moreover, since the set of previous inaccessible cardinals below $\lambda$ does not reflect at its complement, it is impossible for the generic club to avoid ground model singular cardinals of countable cofinality. 
Thus, we restrict our club forcing poset to diagonalize $\mathcal{U}$ only relative to the set of 
the regular cardinals in $V$. 
To make this precise, we denote by $\Sing$ the set of all ground model singular cardinals below $\lambda$, and define
$\bar{\U}= \{ \Sing \cup A \mid A \in \mathcal{U}\}$. We  force with the poset $\co_{\bar{\U}}$, consisting of pairs $(c,B)$ where $c \subseteq \lambda$ is a closed bounded set and $B \in \bar{\U}$. The extension order is as in the previous section.\vskip\bigskipamount

We start by recalling a fundamental and useful fact, which lies in the heart of the proof of the Prikry Property of $\mathbb{P}_E$.
\begin{lemma}\label{lem:extender.properness}
Let $M$ be an elementary sub-model of $H_\chi$ for some large $\chi$ such that $M \cap \lambda = \delta \in \lambda$ is inaccessible cardinal and $M^{<\delta} \subseteq M$. Let $p\in M \cap \mathbb{P}_E$. 

Then, there is a condition $f^* \in \po_E^{*}$ which is $M$-generic (namely, it belongs to every dense open subset of $\po_E^*$ in $M$) and $\dom f^* = M \cap \lambda$. Moreover, if $p^* = \langle f^*, T\rangle$ is a condition in $\po_E$, then there is $T^* \subseteq T$, $E(f^*)$-large such that $T^* \subseteq M$ and $D \in M$ is a dense open subset of $\po_E$ then there is a natural number $n$ such that for every $\langle \nu_0, \dots, \nu_{n-1}\rangle$ in the $n$-th level of $T*$, $p^*_{\langle \nu_0, \dots, \nu_{n-1}\rangle} \in D$. 
\end{lemma}
\begin{proof}
The first claim follows from the closure of $\po_E^*$. Let us focus in the second part. 

Let $f^*$ be as in the lemma. Let $D\in M$ be dense open. For each $\langle \nu_0, \dots, \nu_{n-1}\rangle \in M$ and for each $g\in \po_E^* \cap M$, we can ask whether there is a condition $q \in D$ of the form $\langle h, S\rangle \in D$ such that $h \geq^* g_{\langle \nu_0, \dots, \nu_{n-1}\rangle}$. The set of conditions that decide this statement is dense open and definable in $M$ and thus $f^*$ decides whether there is such extension or not (for each possible $\langle \nu_0, \dots, \nu_{n-1}\rangle$). Let $D_{\vec{\nu}}$ be this set and let us split it into two parts $D_{\vec{\nu}}^0 \cup D_{\vec\nu}^1$ according to the decision, where conditions in $D_{\vec\nu}^0$ are direct extensions that enter $D$ after the non-direct extension.

Let $p^* = \langle f^*, T\rangle$. 
Since a typical point $\nu$ in a measure one tree $T$, associate with the measures $E(f^*)$ is a finite sequence of elements contained in $M$ each has size $|\nu| < \kappa$, we may assume that $T \subseteq M$. There is an extension $q \geq p^*$  in $D$. By the definition of the order of $\po_E$, $q$ is obtained by taking first some Prikry extension and then a direct extension, and therefore the Prikry extension is done using some $\vec{\nu} \in M$.  Thus, for this specific Prikry extension, $f^* \in D_{\vec{\nu}}^0$. We conclude that already $p^*_{\vec{\nu}} \in D$. 

We can now shrink $T$ in order to stabilize the length of the extensions that enter $D$. 
\end{proof}

\begin{lemma}
$\bar{\mathcal{U}}$ extends to a $\lambda$-complete filter in the generic extension by $\po_E$.
\end{lemma}
\begin{proof}
Assume that this is not the case. Since $\kappa$ is singular, the closure of $\bar{\mathcal{U}}$ must drop to some cardinal $\rho < \kappa$. Let $\langle \name{A}_i \mid i < \rho\rangle$ be a sequence of names of elements in $\bar{\mathcal{U}}$ which are forced to have non-measure one intersection.

Using the strong Prikry Property, we can find a sequence of direct extensions $p_i$, and natural numbers $n_i$ such that any $n_i$-length Prikry extension of $p_i$ decides the value of $\name{A}_i$. Since there are fewer than $\lambda$ many such extensions, we can find a set $B_i \in \bar{\mathcal{U}}$ such that $p_i \Vdash B_i \subseteq \name{A}_i$. In particular, $p_{\rho} \Vdash \bigcap B_i \subseteq \bigcap \name{A}_i$, but $\bigcap B_i \in \bar{\mathcal{U}}$. 
\end{proof}

\begin{lemma}
$\co_{\bar{\mathcal{U}}}$ is $\lambda$-distributive in the generic extension by $\po_E$. 
\end{lemma}
\begin{proof}
Since $\kappa$ is singular in the extension by $\po_E$, it is enough to show that the forcing $\co_{\bar{\mathcal{U}}}$ is $\rho$-distributive for every $\rho < \kappa$.%\vskip\bigskipamount

We first work in $V$. 
Let $\name{\vec{D}} = \langle \name{D}_i \mid i < \rho\rangle$ be a sequence of $\po_E$-names for dense open subsets of $\co_{\bar{\mathcal{U}}}$, $\rho < \kappa$. Let $\langle p,q\rangle$ be a condition in $\po_E * \co_{\bar{\mathcal{U}}}$.  Let us define an increasing sequence of models $\langle M_i \mid i < \rho\rangle$ such that:
\begin{itemize}
\item $\name{\vec{D}}, \po_E, \co_{\bar{\mathcal{U}}} \in M_0$.
\item $M_i \prec H_\chi$ for some large $\chi$, $M_i \cap \lambda = \delta_i \in \lambda$ inaccessible. 
\item $M_{i}^{<\delta_{i}} \subseteq M_{i}$ and $\delta_{i} \in \bigcap \{A \in \mathcal{U} \cap M_{i}\}$.
\item $\langle M_j \mid j < i\rangle \in M_i$.
\end{itemize}
This chain of models can be easily obtained using the same argument as in Lemma \ref{Lem01}. %(only the successor steps are nontrivial).

Next, let us pick by induction, for each $i < \rho$, an $M_i$-generic condition $f^*_i\in \po_E^*$ such that $f^*_i \in M_{i+1}$, and $f^*_i \subseteq f^*_j$ for $i < j$. We will define a sequence of names $\name{q}_i$ and a sequence of conditions $p_i$ such that:
\begin{itemize}
\item $p_i = \langle f^*_i, T_i\rangle \in M_{i+1}$, $q_i \in M_i$. 
\item $p_{i+1} \Vdash q_{i+1} \in \name{D}_i$. 
\item The sequence of conditions $p_i$ is $\leq^*$-increasing. Let $p_{\rho}$ be their limit. 
\item $p_{\rho}$  forces that the conditions $q_i$ are increasing and they have a limit $q_\rho$.
\end{itemize}

In $M_{i}$, let $D'_i$ be the dense open set in $\po_E$ of all extensions of $p_i$ that force for some condition $q = (c^q, B^q) \geq q_i$ to be in $\name{D}_i$, and decide its maximum and its large set $B^q$ from $\bar{\mathcal{U}}$. By applying Lemma \ref{lem:extender.properness} inside $M_{i}$, we conclude that there is an $E(f^*_i)$-large tree $T_i\subseteq M_i$ and a natural number $n_i$ such that for the condition $p_{i} = \langle f^*_i, T_i\rangle$, for every $\vec{\nu} \in Lev_{n_i}(T_i)$, $(p_i)_{\vec{\nu}}\in D'_i$. In particular, it picks a condition $q_{i+1, \vec{\nu}} \geq q_i$ from $\co_{\bar{\mathcal{U}}}$, which is going to be in $M_{i}$. Since this condition is in $M_i$, it is going to be bounded below $\delta_{i}$ and its large set belongs to $\bar{\mathcal{U}} \cap M_{i}$. 

Note that the collection of all $n$-step extensions of a fixed condition in $\po_E$ is always a maximal anti-chain above this condition and thus, we can define $\name{q}_{i+1}'$ to be equal to $q_{i+1, \vec{\nu}}$ above $(p_i)_{\vec{\nu}}$, and trivial below any condition which is incompatible with $p_i$.  Finally, we define $\name{q}_{i+1}$ to be the extension of $\name{q}_{i+1}'$ by the single ordinal $\delta_{i}$. By the construction, this is indeed an extension, as $\delta_{i} \in B$ for all $B \in \bar{\mathcal{U}} \cap M_{i}$.

At limit steps, we define $\name{q}_i$ to be the limit of previous conditions. This is possible, since the filter $\bar{\mathcal{U}}$ is still $\lambda$-complete and since the maximal element of the closed set in $\name{q}_j$ is forced to be $\delta_j$ and therefore, the maximal element of $\name{q}_i$ is $\delta_i$ which is singular strong limit cardinal in the limit case.
\end{proof}

\begin{lemma}\label{Lem: preservation of stationarity}
Let $B' \in \mathcal{U}$. Then $B'$ is stationary in $\po_E * \co_{\bar{\mathcal{U}}}$. 
\end{lemma}
\begin{proof}
Let $\name{C}$ be a name for a club. We show that every  condition $q \in \co_{\bar{\U}}$ has an extension which forces that $\name{C} \cap B' \neq \emptyset$. 
Working in $V$, 
let $M \prec H_\chi$, such that $M \cap \lambda = \delta$, $\po_E, \co_{\bar{\mathcal{U}}}, \name{C}, q,B' \in M$, and $\delta \in B'$ is inaccessible. Moreover, let us assume that $M$ is obtained as a union of a chain of models of length $\delta$, $M_i$, such that $M_i \cap \lambda = \delta_i$ and $M_{i+1}^{<\delta_{i+1}} \subseteq M_{i+1}$ and $\delta_{i+1} \in \bigcap (\mathcal{U} \cap M_{i+1})$. 

For each $i$, let $f^*_i$ be $M_i$-generic for $\po_E^*$, such that $f_i^* \subseteq f_j^*$ for $i < j$. Let $f^* = \bigcup f_i^*$.

Let $G \subseteq \po_E$ be a generic filter that contains a condition $p^* = \langle f^*, A\rangle$, $A \subseteq M$. In $V[G]$, $\cf \delta = \omega$. Let $\langle \delta_n \mid n < \omega\rangle$ be a cofinal sequence in $\delta$. For each $n$, for sufficiently large $\xi < \delta$, $M_\xi$ contains the dense set of conditions in $\po_E$ that decide on some condition $q \in \co_{\bar{\mathcal{U}}}$ that forces some ordinal $\gamma_n \geq \delta_n$ to be in $\name{C}$. 

Following the same arguments as in the previous lemma, we can define a condition $\name{q}_n$ by going over some maximal anti-chain. The maximum of the closed set of $\name{q}_n$ is always $\delta_{n + 1}$.
Finally, the sequence of conditions $q_n^G$ has an upper bound, by attaching $\delta$ on top of the union. Let $q_\omega$ be the upper bound. Clearly, $q_\omega$ forces $\delta \in \name{C}$, as wanted. 
\end{proof}

Finally, the following proposition finishes the proof of Theorem \ref{thm:singular.strong.meas}. 

\begin{proposition}
Let $\kappa$ be $\lambda$-supercompact, where $\lambda$ is measurable. Then, there is a generic extension in which $\cf \kappa = \omega$, $\kappa$ is a cardinal, $\lambda = \kappa^+$ and it is $(\left(S_{reg}^\lambda\right)^V, 1)$-strongly measurable cardinal. 
\end{proposition}

We can now finish the proof of theorem \ref{thm:singular.strong.meas}.
\begin{proof}[Proof of Theorem \ref{thm:singular.strong.meas}]%\vskip\bigskipamount
The iteration $\po_E * \co_{\bar{\mathcal{U}}}$ is cone homogeneous as an iteration of two cone homogeneous, ordinal definable forcing notions. Since $\po_E$ preserves cardinals below $\kappa$ and $\geq\lambda$ and $\co_{\bar{\mathcal{U}}}$ preserves cardinals, the result follows. The set $S_{reg}^\lambda$ is stationary by Lemma \ref{Lem: preservation of stationarity}.
\end{proof}

The result that we obtain for the successor of a singular cardinal is weaker than the result for a successor of a regular cardinal. The reason is that in order to get the closed unbounded filter to be sets from the intersection of some ground model normal measures we will have to obtain a situation in which the regular cardinals between the supercompact cardinal $\kappa$ and the measurable cardinal $\lambda$ are going to change cofinalities into values which differ from the cofinality of $\kappa$ in the generic extension. This is also the reason that such a method cannot work for getting $\omega$-strongly measurable successor of a singular cardinal of uncountable cofinality.

We remark that Woodin in \cite{Woodin-SuitableExtendersI}, proved that it is consistent relative to the large cardinal axiom $I_0$ that a successor of a singular cardinal is $\omega$-strongly measurable.
\begin{question}
Is it consistent that there is an $\omega$-strongly measurable cardinal $\lambda^+$, where $\cf \lambda > \omega$ is a limit cardinal?
\end{question}
\begin{question}
Is it consistent that there is a cardinal $\lambda^+$, where $\cf \lambda > \omega$ is a limit cardinal and $\left(S^{\lambda^{+}}_{reg}\right)^{\HOD}$ contains a club in $V$?
\end{question}

\section{Appendix - Homogeneity}\label{appendix:homogeneity}
In this section we review some basic facts related to homogeneity and develop some basic tools in order to preserve homogeneity of iterations of Prikry type forcings.
\subsection{Homogeneity and $\HOD$}\label{Sec:homog}
When dealing with $\HOD$, we would like to modify the universe (via forcing) while not adding objects to $\HOD$. The main method to obtain this is to force with posets which satisfy certain weak homogeneity property. The main results of this work will focus on the notion of cone homogeneous posets. 

\begin{definition}\label{Def:ConeHom}
 We say that a poset $\po$ is \textbf{cone homogeneous} if for every $p,q \in \po$ there are extensions $p^*,q^*$ of $p,q$ respectively, and a forcing isomorphism $\varphi$ from the cone $\po/p^*$ (i.e., of conditions extending $p^*$) to the cone $\po/q^*$.
\end{definition}

This notion can also be found under different names in the literature concerning weak forms of homogeneity. Our terminology follow Dorbinen and Friedman, \cite{DrobinenFriedman}, for the most part. It is easy to see that cone homogeneous posets satisfy most standard properties of homogeneous posets concerning ordinal definability sets. In particular, the following well-known result holds.

\begin{fact}[Levy, \cite{Levy}]\label{fact:Levy}
If $\po$ is cone homogeneous and belongs to $\HOD$, and $G \subseteq \po$ is generic over $V$ then $\HOD^{V[G]} \subseteq \HOD^V$. 
\end{fact}

If $\varphi$ is an isomorphism of two cones $\po /p_0$ and $\po/p_1$ and $\sigma$ is a $\po /p_0$ name, then by recursively applying $\varphi$ we obtain a $\po /p_1$-name, which we denote by $\sigma^\varphi$.

Let $\po = \po_\kappa$ where $\l \po_\alpha,\qo_\alpha \mid \alpha < \kappa\r$ is an iteration of cone homogeneous posets $\qo_\alpha$ and moreover let us assume that all cone automorphisms of $\po_\alpha$ do not modify $\qo_\alpha$ as a poset. For simplicity, we may assume that $\qo_\alpha$ and its order are ordinal definable.

Given two conditions $\vec{p} = \l p_\alpha \mid \alpha < \kappa\r, 
 \vec{q} = \l q_\alpha \mid \alpha < \kappa\r$ in $\po$, it is natural to try forming extensions $\vec{p^*} = \l p^*_\alpha \mid \alpha < \kappa \r \geq \vec{p}$, $\vec{q^*} = \l q^*_\alpha \mid \alpha < \kappa \r \geq \vec{q}$, and an isomorphism $\varphi\colon \po/\vec{p^*} \to \po/\vec{q^*}$ as follows:

 By induction on $\beta \leq \kappa$, we attempt defining extensions $\vec{p}^\beta = \l p^*_\alpha \mid \alpha < \beta\r $ of  $\vec{p}\uhr \beta$, and $\vec{q}^\beta = \l q^*_\alpha \mid \alpha < \beta\r$ of $\vec{q}\uhr \beta$, and an isomorphism $\varphi_\beta \colon \po_\beta/\vec{p}^\beta \to \po_\beta/\vec{q}^\beta$. 
 Our inductive assumptions further include $\vec{p}^{\beta_1}\uhr \beta_0 = \vec{p}^{\beta_0}$, 
 $\vec{q}^{\beta_1}\uhr \beta_0 = \vec{q}^{\beta_0}$, and $\varphi_{\beta_1} \uhr \po_{\beta_0}/\vec{p}^{\beta_0} = \varphi_{\beta_0}$,\footnote{i.e., the restriction 
 $\varphi_{\beta_1} \uhr \po_{\beta_0}/\vec{p}^{\beta_0}$ is obtained by identifying conditions $\vec{r}^{\beta_0} \in \po_{\beta_0}/\vec{p}^{\beta_0}$ with their extension 
 $\vec{r}^{\beta_1} = \vec{r}^{\beta_0} \fr (\vec{p}^{\beta_1}\uhr_{[\beta_0,\beta_1)})$.}
 for all $\beta_0 < \beta_1$.
 
For $\beta=0$, where $\po_0 = \{ 0_{\po_0}\}$ is a trivial forcing we take $\varphi_0$ to be the identity. 
At a successor step, assuming $\vec{p}^\beta,\vec{q}^\beta$ and $\varphi_\beta$ have been defined, we have that $\varphi_\beta(\vec{p}^\beta) = \vec{q}^\beta$ forces that
$p_\beta^{\varphi_\beta}$ and $q_\beta$ are conditions of the cone homogeneous poset $\qo_\beta$. There are therefore $\po_\beta$-names 
$p'_\beta$ and $q'_\beta$ of extensions of $p_\beta^{\varphi_\beta}$ and $q_\beta$, respectively, and a name of a cone isomorphism 
$\psi_\beta \colon \qo_\beta/p'_\beta \to \qo_\beta/q'_\beta$. We stress that we use the maximality principle, and do not extend the conditions $\vec{p}^\beta$ and $\vec{q}^\beta$   in order to determine the values of $p'_\beta, q'_\beta$ and $\psi_\beta$.

Let $p^*_\beta = (p'_\beta)^{\varphi_\beta}$ and $q^*_\beta = q'_\beta$. Clearly, $\vec{p}^\beta$,$\vec{q}^\beta$ force that $p^*_\beta,q^*_\beta$ extend $p_\beta, q_\beta$, respectively. We set $\vec{p}^{\beta+1} = \vec{p}^\beta \fr \l p^*_\beta\r$, $\vec{q}^{\beta+1} = \vec{q}^\beta \fr \l q^*_\beta\r$, and define 
$\varphi_{\beta+1}\colon \po_{\beta+1}/\vec{p}^{\beta+1} \to \po_{\beta+1}/\vec{q}^{\beta+1}$ by mapping a condition 
$\vec{r} = \vec{r}\uhr\beta \fr \l r_\beta\r \in \po_{\beta+1}/\vec{p}^{\beta+1}$ to 
\[
\varphi_{\beta+1}( \vec{r}) = \varphi_\beta(\vec{r}\uhr\beta) \fr \l \psi_\beta(r_\beta^{\varphi_\beta})  \r.
\]
It is immediate from our assumption of $\varphi_{\beta}$ and choice of $\psi_\beta$ that $\varphi_{\beta+1}$ is an isomorphism.
Finally, for a limit ordinal $\delta \leq \kappa$, $\vec{p}^{\delta}$ (similarly  $\vec{q}^\delta$) is determined by the requirement 
$\vec{p}^{\delta}\uhr\beta = \vec{p}^\beta$ for all $\beta < \delta$ (similarly for $\vec{q}^{\delta}$), 
and $\varphi_\delta$ by the requirement $\varphi_\delta\uhr \po_\beta/\vec{p}^\beta = \varphi_\beta$ for all $\beta < \delta$. See \cite{DrobinenFriedman} for more detailed proof for the validity of this construction. \vskip\medskipamount

We conclude that for this construction to succeed the following conditions need to hold for all $\beta\leq \kappa$:
(i) $\vec{p}^\beta,\vec{q}^\beta$ are well-defined conditions in $\po_\beta$ which extend $\vec{p}\uhr\beta,\vec{q}\uhr\beta$ respectively, 
and (ii) $\varphi_\beta$ is a well-defined cone isomorphism. \vskip\smallskipamount
If the construction succeeds throughout all stages $\beta \leq \kappa$, then the final conditions $\vec{p^*} = \vec{p}^\kappa$, 
$\vec{q^*} = \vec{q}^\kappa$ and cone isomorphism $\varphi = \varphi_\kappa$, satisfy the required properties. 
It is easy to see that condition (i) and (ii) may only fail at limit stages $\delta \leq \kappa$, where the precise formation of the iteration (e.g., its support) 
may prevent $\vec{p}^\delta$ to be a condition in $\po_\delta$. 
Similarly, the definition of the limit order $\leq_{\po_\delta}$ might prevent the defined map $\varphi_\delta$ to be an isomorphism.

This problem does not occur for finite iteration:
\begin{lemma}[{\cite{DrobinenFriedman}}]\label{corr: OD implies homogeneous}
A finite iteration of ordinal definable cone homogeneous forcings is cone homogeneous. 
\end{lemma}

Since our proof of theorem \ref{THM1} is based on a construction of a \textbf{Magidor Iteration} $\po = \l \po_\alpha,\qo_\alpha \mid \alpha < \theta\r$ of Prikry-type forcings $(\qo_\alpha,\leq_{\qo_\alpha},\leq^*_{\qo_\alpha})$, we conclude this section with a description of a specific variant of cone homogeneity for the posets $\qo_\alpha$, which guarantees that the Magidor iteration $\po$ is cone-homogeneous as well. 

\begin{definition}[{Prikry type forcing, \cite{Gitik-HB}}]
$\langle \mathbb{P}, \leq, \leq^*\rangle$ is a Prikry type forcing if 
\begin{itemize}
\item $\leq \supseteq \leq^*$ are partial orders on $\mathbb{P}$ and 
\item (the Prikry Property) for every statement $\sigma$ in the forcing language for $\langle \mathbb{P}, \leq\rangle$, and a condition $p$ there is a condition $p^*$, $p \leq^* p^*$ such that $p^* \Vdash \sigma$ or $p^* \Vdash \neg \sigma$.
\end{itemize}
\end{definition}

Conditions in the Magidor iteration $\po = \l \po_\alpha,\qo_\alpha \mid \alpha < \kappa\r$ of Prikry type posets $\l \qo_\alpha, \leq_{\qo_\alpha},\leq^*_{\qo_\alpha}\r$  are sequences $\vec{p} = \l p_\alpha \mid \alpha < \kappa\r$ which beyond the standard requirement of $p\uhr\alpha \Vdash p_\alpha \in \qo_\alpha$, also satisfy that for all but finitely many ordinals $\alpha < \kappa$, $\vec{p} \uhr \alpha \Vdash p_\alpha \geq^*_{\qo_\alpha} 0_{\qo_\alpha}$. We note that in particular, the definition allows using full-support condition, as long as almost all components $p_\alpha$ are direct extensions of the trivial conditions.  Similarly for the definition of the ordering $\leq_{\po}$, we have that $\vec{p'} \geq \vec{p}$  requires 
both that $\vec{p'}\uhr\alpha \Vdash p'_\alpha \geq_{\qo_\alpha} p_\alpha$ for all $\alpha$ and that 
 for all but finitely many ordinals $\alpha < \kappa$, $\vec{p'}\uhr \alpha \Vdash p'_\alpha \geq^* p_\alpha$.
 See \cite{Gitik-HB} for a comprehensive description of the Magidor iteration style and its main properties.

\begin{lemma}\label{FACT:homogiter}
 Suppose that $\po = \l \po_\alpha,\qo_\alpha \mid \alpha < \kappa\r$ is a Magidor iteration of Prikry-type posets $\l \qo_\alpha, \leq_{\qo_\alpha},
 \leq^*_{\qo_\alpha}\r$ so that the following conditions hold for each $\alpha < \kappa$:
 \begin{itemize}
  \item[(i)]$\qo_\alpha$, $\leq_{\qo_\alpha}$ and $\leq^*_{\qo_\alpha}$ are ordinal definable in $V$, and
  \item[(ii)] it is forced by $0_{\po_{\alpha}}$ that for every two conditions $p,q \in \qo_\alpha$ there are $p^* \geq_{\qo_\alpha}^* p$ and 
  $q^* \geq^*_{\qo_\alpha} q$ and a cone isomorphism $\psi_\alpha\colon \qo_\alpha/p^* \to \qo_\alpha/q^*$ which respects the direct extension order $\leq^*_{\qo_\alpha}$.
 \end{itemize}
Then $\po$ is cone homogeneous. 
\end{lemma}
\begin{proof}
Let $\vec{p},\vec{q} \in \po$, and $(\vec{p}^\beta,\vec{q}^\beta,\varphi_\beta \mid \beta \leq \kappa)$ be the sequence obtained form the procedure described above. 
 It suffices to verify inductively, that conditions (i) and (ii) are satisfied by the sequence. 
 
 We note that in the successor step construction of $p'_\beta$, $p^*_\beta = (p'_\beta)^{\varphi_\beta^{-1}}$,  $q^*_\beta = q'_\beta$, and $\psi_\beta$,
 we may assume that $\varphi_\beta(\vec{p}^\beta) \Vdash p'_\beta \geq^*_{\qo_\beta} p_\beta^{\varphi_\beta}$, $\vec{q}^\beta\Vdash q'_\beta \geq^* q_\beta$, and that $\psi_\beta$ is $\leq^*_{\qo_\beta}$-preserving. Since $\leq^*_{\qo_\beta}$ is ordinal definable in $V$, $0_{\mathbb{P}_{\beta}} \Vdash \leq^*_{\qo_\beta} = (\leq^*_{\qo_\beta})^{\varphi_\beta^{-1}}$, and therefore by applying the automorphism $\varphi_\beta^{-1}$ we get 
 $\vec{p}^\beta \Vdash p^*_\beta \geq^*_{\qo_\beta} p_\beta$, and $p \mapsto \psi_{\beta}(p^{\varphi_\beta})$ is forced by $\vec{p}^\beta$ to be $\leq^*_{\qo_\beta}$-preserving in the cone below $p^*_\beta$. 
 In particular, assuming $\varphi_\beta$ is order preserving and  $\vec{r}^{\beta+1} \geq \vec{s}^{\beta+1} \in \po_{\beta+1}/\vec{p}^{\beta+1}$, 
 $\vec{r}^{\beta+1} = \vec{r}^\beta \fr \l r_\beta\r$, $\vec{s}^{\beta+1} = \vec{s}^\beta \fr \l s_\beta\r$,
 we have that if $\vec{r}^\beta \Vdash r_\beta \geq^*_{\qo_\beta} s_\beta$ then 
 $\varphi_{\beta}(\vec{r}^\beta) \Vdash \psi_{\beta}(r_\beta^{\varphi_\beta}) \geq^*_{\qo_\beta} \psi_{\beta}(r_\beta^{\varphi_\beta}) = \varphi_{\beta+1}(\vec{s}^{\beta+1})_\beta$.
 The same conclusion holds for $\leq_{\qo_\beta}$ which is also ordinal definable in $V$. 
 \vskip\bigskipamount
 
 We conclude that, first, $p^*_\beta, q_\beta^*$ are forced to be direct extensions of $0_{\qo_\beta}$ whenever $p_\beta,q_\beta$ are, which in turn, implies that 
 $\vec{p}^\alpha,\vec{q}^\alpha$ are conditions of $\po_\alpha$ for all $\alpha \leq \kappa$. Hence, (i) is satisfied. 
 Second, for every $\vec{r}^\alpha,\vec{s}^\alpha \in \po_\alpha/\vec{p}^\alpha$ and $\beta < \alpha$, if 
$\vec{r}^\alpha \uhr \beta \Vdash r_\beta \geq^*_{\qo_\beta} s_\beta$
 then $\varphi_\alpha(\vec{r}^\alpha)\uhr\beta \Vdash \varphi_\alpha(\vec{r}^\alpha)_\beta \geq^*_{\qo_\beta} \varphi_\alpha(\vec{s}^\alpha)_\beta$, and similarly, when replacing replacing $\leq^*_{\qo_\beta}$ with $\leq_{\qo_\beta}$.
 It follows at once from this and the definition of the ordering $\leq_{\po_\alpha}$ of the Magidor iteration that $\varphi_\alpha$ is a cone isomorphism. Hence (ii) holds. 
\end{proof}
\subsection{Homogeneous change of cofinalities}\label{ssec-non-stationary-iteration}

Our approach to construct a model with an $\omega$-strongly measurable cardinal $\kappa$, is to force over a ground model satisfying $V = \HOD$ with a weakly homogeneous poset (i.e., therefore also cone-homogeneous) to form a generic extension $V[G]$ with 
a cardinal $\kappa$, which satisfies the conditions of lemma \ref{lem:sufficient.omega.strongly}.
In light of lemma \ref{lem:necessary.omega.strongly} above, we see that many regular cardinals in $V$ need to change their cofinality in $V[G]$. 
The main challenge in that regard, is to change the cofinality of many cardinals with a weakly homogeneous forcing. 

Fortunately, such forcing has been constructed in \cite{bunger}, where the theory of non-stationary support iteration of Prikry-type forcings is developed, and employed to form a weakly-homogeneous variant of the Gitik iteration (\cite{gitik-nonstionary-ideal}). We note that as opposed to an Easton-style version of the Gitik iteration, which has a good chain condition (i.e., $\kappa$-c.c. when iterating up to a Mahlo cardinal $\kappa$), the non-stationary support variant of \cite{bunger} has a weaker, fusion-type property. \vskip\bigskipamount

We briefly describe the construction of the non-stationary support iteration $\po$ of iteration of Prikry-type forcings $\qo_{\alpha}$ from  \cite{bunger}.
The iteration, which is based on the given coherent sequence of measures $\l U_{\alpha,\tau}\mid \alpha <\kappa, \tau <o^{\U}(\alpha)\r$ is nontrivial at each $\alpha< \kappa$, $o^{\U}(\alpha) > 0$. 
As this $\alpha$, the forcing $\qo_\alpha$ adds a cofinal closed unbounded set $b_{\alpha}$ to $\alpha$ of order-type $\omega^{o(\alpha)}$ (ordinal exponentiation).
More specifically, given a $V$-generic filter $G_\alpha \subseteq\po_\alpha$, which adds clubs $b_\beta$, for $\beta< \alpha$, $o^{\U}(\beta) > 0$, 
one considers finite sequences $t = \l\nu_0,\dots,\nu_{k-1}\r$ with the property that for every $i < k-1$, if $o^{\U}(\nu_i) < o^{\U}(\nu_{i+1})$ then 
$\nu_i \in b_{\nu_{i+1}}$ and $b_{\nu_{i+1}} \cap \nu_i = b_{\nu_i}$. Such sequences are called coherent (with respect to $G_\alpha$). If $\rho$ is an ordinal so that $o(\nu_i) < \rho$ 
for all $i < k$ then we say $t$ is $\rho$-coherent. Otherwise, we denote by $t\uhr \rho$ to be the sub-sequence of $\nu_i \in t$ so that $o(\nu_i) < \rho$. 

Working at $V[G_\alpha]$, one constructs posets $\qo_{\alpha,\tau}$, $\tau \leq o^{\U}(\alpha)$, and simultaneously shows by induction on $\tau \leq o^{\U}(\alpha)$ that 
for each $\tau$-coherent sequence $t$,  $U_{\alpha,\tau} \in \U$ extends to $U_{\alpha,\tau}(t)$.
We define $\qo^0_{\alpha}$ to be the trivial poset, and 
given that the measures $U_{\alpha,\tau'}(t')$, have been defined for every $\tau' < \tau$, and every $\tau'$-coherent sequence $t'$, 
the forcing $\qo^\tau_{\alpha}$ consists of pairs $q = \l t, T\r$ where $t$ is $\tau$-coherent, $T \subseteq[\alpha]^{<\omega}$ is a tree whose stem is $\emptyset$
and for every $s \in T$, $\suc_T(s) := \{ \mu < \kappa \mid s \fr \l \mu \r \in T\} \in \bigcap_{\tau' < \tau} U_{\alpha,\tau'}( t\fr s\uhr\tau')$.
Direct extensions and end extensions of $\qo_{\alpha}^\tau$ are defined as usual. $\qo_{\alpha}^\tau$ is a Prikry type forcing whose direct extension is $\alpha$-closed. 
With $\qo_{\alpha,\tau}$ determined we consider the $V$-ultrapower by $U_{\alpha,\tau}$,
by taking $j_{\alpha,\tau}\colon V \to M_{\alpha,\tau} \cong Ult(V,U_{\alpha,\tau})$, and define for each $\tau$-coherent sequence $t$ a $V[G_\alpha]$ measure
$U_{\alpha,\tau}(t)$ by $X = \name{X}_{G_\alpha} \in U_{\alpha,\tau}(t)$ if there exist $p \in G_\alpha$ and a valid tree $T$ such that 
\[ p \fr \l t,T\r \fr j_{\alpha,\tau}(p)\setminus(\alpha+1) \Vdash_{j_{\alpha,\tau}(\po)} \check{\alpha} \in j_{\alpha,\tau}(\name{X}).\]
\begin{fact}\label{fact:summary-properties-of-PU} ${}$
\vskip\smallskipamount
\begin{enumerate}
\item For each $\alpha$ such that $o^{\mathcal{U}}(\alpha) > 0$, $b_\alpha$ is a cofinal sequence at $\alpha$ of order type $\omega^{o^{\mathcal{U}}(\alpha)}$ (ordinal exponentiation). 
  \item For each $\alpha \leq \kappa$, $(\po_\alpha,\leq,\leq^*)$ is a Prikry-type forcing. 
  \item For every $\gamma < \alpha \leq \kappa$, the quotient $(\po_\alpha/\po_\gamma, \leq,\leq^*)$ is a Prikry-type forcing whose direct extension order $\leq^*$ is $\gamma$-closed. In particular, the quotient $\po_\alpha/\po_\gamma$ does not add new bounded subsets to $\gamma$. 
\item For every $\gamma < \alpha$, the iteration $\mathbb{P}_\alpha / \mathbb{P}_{\gamma+1}$ is weakly homogeneous.
\end{enumerate}
\end{fact}

\section{acknowledgments}
We are grateful to Sandra M\"{u}ller and Grigor Sargsyan for valuable conversations concerning the Inner Model Program and the study of $\HOD$.   
%\newpage
\providecommand{\bysame}{\leavevmode\hbox to3em{\hrulefill}\thinspace}
\providecommand{\MR}{\relax\ifhmode\unskip\space\fi MR }
\providecommand{\MRhref}[2]{%
  \href{http://www.ams.org/mathscinet-getitem?mr=#1}{#2}
}
\providecommand{\href}[2]{#2}

\end{document}